\documentclass[english,10pt]{article}
\usepackage[english]{babel}
\usepackage{amsfonts}
\usepackage{amssymb}
\usepackage{amsmath}
\usepackage{amsthm}
\usepackage{multirow}
\usepackage{lscape}
\usepackage{graphicx}%
\newcommand{\field}[1]{\mathbb{#1}}
\newcommand{\R}{\field{R}}

\newcommand{\N}{\field{N}}
\newtheorem{theorem}{{\sc{Theorem}}}
\newtheorem*{nntheorem}{{\sc{Theorem}}}

\newtheorem{corollary}[theorem]{\sc{Corollary}}
\newtheorem{definition}[theorem]{\sc{Definition}}
\newtheorem{example}[theorem]{\sc{Example}}
\newtheorem{lemma}[theorem]{\sc{Lemma}}

\newtheorem{proposition}[theorem]{\sc{Proposition}}

\newtheorem{remark}[theorem]{\sc{Remark}}

    \makeatletter
    \let\@fnsymbol\@arabic
    \makeatother

\begin{document}

\title{Integrability on Direct Limits of Banach Manifolds}

\author{Patrick CABAU $^1$
\& Fernand PELLETIER $^2$}

\maketitle

\begin{abstract}
In this paper, we study several objects in the framework of direct limits of anchored Banach bundles over particular convenient manifolds (direct limits of Banach manifolds).
In particular, we give a criterion of integrability for distributions on such
convenient manifolds which are locally direct limits of particular sequences
of Banach anchor ranges.
\end{abstract}

\textbf{MSC 2010}.- 58A30, 18A30, 46T05; secondary 17B66, 37K30, 22E65. \\
\medskip
\textbf{Keywords}.- Integrable distribution; direct limit; almost Lie
Banach algebroid; almost Lie bracket; Koszul connection; anchor range.

\medskip
$^1$ Lyc\'{e}e Pierre de Fermat, Parvis des Jacobins, 31000 Toulouse, France\\
Patrick.Cabau@ac-toulouse.fr
\smallskip \\
$^2$ Lama, Universit\'{e} de Savoie Mont Blanc, 73376 Le Bourget du Lac Cedex, France \\
fernand.pelletier@univ-smb.fr
 \bigskip

\textbf{Acknowledgments}.- The authors would like to thank the anonymous reviewer for providing valuable comments and suggestions.

\section{Introduction and results}

In classical differential geometry, a \textit{distribution} on a smooth
manifold $M$ is an assignment
\[
\mathcal{D}:x\mapsto\mathcal{D}_{x}\subset T_{x}M
\]
on $M$, where $\mathcal{D}_{x}$ is a subspace of $T_{x}M$. This distribution
is \textit{integrable} if, for any $x\in M$, there exists an immersed
submanifold $f:L\rightarrow M$ such that $x\in f(L)$ and for any $z\in L$, we
have $Tf(T_{z}L)=\mathcal{D}_{f(z)}$. On the other hand, $\mathcal{D}$ is
called \textit{involutive} if, for any vector fields $X$ and $Y$ on $M$
tangent to $\mathcal{D}$, their Lie bracket $[X,Y]$ is also tangent to
$\mathcal{D}$.

On a finite dimensional manifold, when $\mathcal{D}$ is a subbundle of $TM$, the
classical Frobenius Theorem gives an equivalence between integrability and
involutivity. In the other case, the distribution is \textit{singular} and,
even under assumptions of smoothness on $\mathcal{D}$, in general, the
involutivity is not a sufficient condition for integrability (one needs some
more additional local conditions). These problems were clarified and resolved
essentially in \cite{Sus} and \cite{Ste}.

In the context of Banach manifolds, the Frobenius Theorem is again true for
distributions which are complemented subbundles in the tangent bundle. For
singular Banach distributions closed and complemented (i.e. $\mathcal{D}_{x}$
is a complemented Banach subspace of $T_{x}M$) we also have the integrability
property under some natural geometrical conditions (see \cite{ChSt} for
instance). In a more general way,  for weak Banach distributions $\mathcal{D}$, the integrability property is again true under some geometrical
criterion (see \cite{Pel} for more details).

\bigskip

The notion of \textit{Lie algebroid} $\mathcal{A}=\left(  E,\pi,M,\rho,
[\;,\;] _{E}\right)  ,$ where $\pi:E\longrightarrow M$ is a fiber bundle and
where the anchor $\rho$ is a morphism of Lie algebras, was first introduced by
Pradines in \cite{Pra}. Such objects can be seen as generalizations of both
Lie algebras and tangent vector bundles. This context is an adapted framework
for different problems one can meet in Mechanics (e.g. non holonomic
lagrangian systems, \cite{CLLM}) or in symplectic Geometry in view of the
symplectization of Poisson manifolds and applications to quantization
(\cite{Kar}, \cite{Wei}).

\bigskip

The Stefan-Sussmann's  Theorem implies the integrability of the distribution
$\rho(E)$ for a  finite dimensional Lie algebroid $\left(  E,\pi,M,\rho,[\;,\;]
_{E}\right)  $. Moreover one also gets the existence of symplectic leaves for
Lie-Poisson Banach manifolds under comparable assumptions. For a Banach   Lie algebroid $\left(  E,\pi,M,\rho,[\;,\;]
_{E}\right)$ the same result is also true under some additional assumptions (see \cite{Pel}).

\bigskip

However, the Banach context is not necessarily the most appropriate: for
instance, in the framework of Lie-Poisson structure on the dual of the Lie
algebra of an infinite-dimensional Lie group, the adapted model is not
anymore the Banach one.  A lot of infinite-dimensional Lie
groups $G$, linked with symmetries depending on infinitely many parameters one
can meet in Mathematical Physics, can often be expressed as the union of an
ascending sequence $G_{1}\subset G_{2}\subset\cdots\subset G_{i}\subset\cdots$
of finite or infinite-dimensional Lie groups. Various examples of such objects
can be found in papers of Gl\"{o}ckner (see \cite{Glo1},
\cite{Glo2} and \cite{Glo3}). The convenient setting  as
defined by \cite{FrKr} and \cite{KrMi}   seems well adapted to this framework (see for instance \cite{Glo1}).

\bigskip

The context of this paper concerns the study of direct limits of anchored
Banach bundles over direct limits of Banach manifolds endowed with convenient
structures and the results of \cite{Glo1}, \cite{Glo2} and \cite{Glo3}. More
precisely, \textit{essentially } we consider sequences $(E_{n},\pi_{n},M_{n},\rho
_{n})_{n\in\mathbb{N}^{\ast}}$ of anchored Banach bundles where $(E_{n},\pi_{n},M_{n})_{n\in\mathbb{N}^{\ast}}$ is a strong ascending sequence of Banach bundles (cf. Definition \ref{D_AscendingSequenceBanachVectorBundles}) where the anchors $\rho_{n}$ and the bonding maps $\lambda_n^m:E_n\longrightarrow E_m$ and $\epsilon_n^m:M_n\longrightarrow M_m$ fulfill conditions of compatibility given in Definition \ref{D_DirectSequenceBanachLieAlgebroids}, (2).

Given such a sequence $(E_{n},\pi_{n},M_{n},\rho_{n})_{n\in\mathbb{N}^{\ast}}$, we get
an anchored convenient bundle $\left(  E=\underrightarrow{\lim}E_{n}, \pi=\underrightarrow{\lim}\pi_{n},M=\underrightarrow{\lim}M_{n},\rho=\underrightarrow{\lim}\rho_{n}\right)  $
(cf. Theorem \ref{T_DLBanachLieAlgebroids_ConvenientLieAlgebroids}). Note
that, according to Gl\"{o}ckner's results, in order to get an interesting
(convenient) structure on the direct limit of Banach manifolds, an essential
hypothesis is \textit{the existence of direct limit charts} (cf. Definition
\ref{DLchart}). In particular, this assumption is true if each member $M_{n}$
of the ascending sequence $M_{1}\subset\cdots\subset M_{n}\subset\cdots$ of
Banach manifolds can be endowed with a Koszul connection $\nabla^{n}$ (cf. Proposition \ref{P_StrongAscending}).\\

Another problem in the context of direct limit of an ascending sequence $\{X_n\}_{n\in \N^*}$ of topological spaces is the following one: even if each $X_n$ is a Hausdorff topological space the direct limit $X=\underrightarrow{\lim}X_n$, provided with the direct limit topology, can be  \textit{not Hausdorff}. This leads us to introduce the notion of \textit{non necessary Hausdorff convenient manifold structure} (cf. Definition \ref{nnhconvenient}).\\

When each $E_{n}$ can be endowed with an almost Lie bracket $[\;,\;]_{n}$
(resp. a Koszul connection $\nabla^{n}$) such that the restriction of
$[\;,\;]_{n+1}$ (resp. $\nabla^{n+1}$) to $E_{n}$ is $[\;,\;]_{n}$ (resp.
$\nabla^{n}$) we obtain an almost Lie bracket $[\;,\;]=\underrightarrow{\lim
}[\;,\;]_{n}$ (resp. a Koszul connection $\nabla=\underrightarrow{\lim}%
\nabla^{n}$) on $E$. Moreover, if for each $n\in\mathbb{N}$ we have $[\rho
_{n}(X),\rho_{n}(Y)]=\rho_{n}[X,Y]_{n}$ (resp. $[\;,\;]_{n}$ satisfies the
Jacobi identity) the same property is true for the direct limit $[\;,\;]$.

Now, according to Theorem 5 of \cite{Pel}, we obtain the main result of this paper:

\begin{nntheorem}
Criterion of integrability (cf. Theorem
\ref{T_IntegrabilityDLKoszulBanachBundles}).\\ Let $\Delta$ be a
distribution on a convenient manifold $M$ with the following properties:

(1) for any $x\in M$, there exists an open neighborhood $U$ of $x$, a strong ascending
sequence of anchored Banach bundles $(E_{n},\pi_{n},U_{n},\rho_{n}%
)_{n\in\mathbb{N}^{\ast}}$ endowed with a Koszul connection $\nabla^{n}$, such
that $U=\underrightarrow{\lim}U_{n},$ $\underrightarrow{\lim}\rho_{n}%
(E_{n})=\Delta_{|U}$ and such that $E_{n}$ is a complemented subbundle of
$E_{n+1}$;

(2) there exists an almost Lie bracket $[\;,\;]_{n}$ on $(E_{n},\pi_{n}%
,U_{n},\rho_{n})$ such that:
\begin{itemize}
\item $(E_{n},\pi_{n},U_{n},\rho_{n},[\;,\;]_{n})$ is a Banach Lie algebroid;
\item over each point $y_{n}\in U_{n}$ the kernel of $\rho
_{n}$ is complemented in the fiber $\pi_{n}^{-1}(y_{n})$.
\end{itemize}
Then the distribution $\Delta$ is integrable and each maximal integral manifold
satisfies the direct limit chart property at any point and so is endowed with
a non necessary Hausdorff convenient manifold structure.
\end{nntheorem}

Note that in the framework of finite-dimensional or Hilbert manifolds, this
criterion of integrability requires much weaker assumptions (cf. Corollary
\ref{C_DLHilbert}). \bigskip

In order to make this article as self-contained as possible we first recall
various notions: the convenient differential calculus setting as defined by
Fr\"{o}licher, Kriegl and Michor (part \ref{*ConvenientDifferentialCalculus}),
direct limits of topological vector spaces (part
\ref{*DL_TopologicalVectorSpaces}) or manifolds (part \ref{*DL_Manifolds}) and
linear connections on Banach bundles (part
\ref{*LinearConnectionsOnDLAnchoredBanachBundles}). In part
\ref{*DL_SequencesAlmostBanachLieAlgebroids}, Theorem
\ref{T_DLBanachLieAlgebroids_ConvenientLieAlgebroids}, we prove that certain
limits of Almost Lie Banach algebroids can be endowed with a structure of
Almost Lie convenient algebroid. In the last part, we prove the previous
theorem which is a criterion of integrability for distributions and we give an
application to actions of direct limits of Banach Lie groups.

\section{\label{*ConvenientDifferentialCalculus}Convenient differential calculus}

Differential calculus in infinite dimensions has already a long history which
goes back to the beginnings of variational calculus developed by Bernoulli and
Euler. During the last decades, a lot of theories of differentiation have been
proposed in order to differentiate in spaces more general than Banach ones;
the traditional calculus for Banach spaces is not satisfactory for the
categorical point of view since the space $C^{\infty}\left(  E,F\right)  $ of
smooth maps between Banach spaces is no longer a Banach space.

The setting of convenient differential calculus discovered by A.
Fr\"{o}licher and A.\ Kriegl (see \cite{FrKr}) is chosen. The reference for this section
is the tome \cite{KrMi} which includes some further results.

\bigskip

In order to define the smoothness on locally convex topological vector spaces
(l.c.t.v.s.) $E$, the basic idea is to test it along smooth curves (cf.
Definition \ref{D_SmoothMapBetweenLCTVS}), since this notion in this realm is
a concept without problems.

\medskip

On the one hand, a curve $c:\mathbb{R}\longrightarrow E$ is \textit{differentiable} if, for all $t$, the derivative $c^{\prime}\left(
t\right)  $ exists where  $c^{\prime}\left(  t\right)  =\underset{h\longrightarrow0}{\lim}\dfrac{1}{h}\left(  c\left(  t+h\right)  -c\left(
t\right)  \right)  $. It is \textit{smooth} if all iterative derivatives
exist.

On the other hand, if $J$ is an open subset of $\R$, we say that $c$ is \textit{Lipschitz} on $J$ if the
set $\{\displaystyle\frac{c(t_2) -c(t_1)}{t_2-t_1} ; t_1,t_2 \in J,  t_1\not= t_2\}$ is bounded in $E$.
The curve $c$ is \textit{locally Lipschitz} if every point in $\R$ has a neighborhood on which $c$ is Lipschitz.
 For $k\in \N$, the curve $c$ is of class $Lip^k$ if $c$ is derivable up to order $k$, and if  the $k^{th}$-derivative $c:\mathbb{R}\longrightarrow E$  is locally Lipschitz.

We then have the following link between both these notions (\cite{KrMi} section 1.2):

\begin{proposition}\label{lipschitz-smooth} Let $E$ be a l.c.t.v.s, and let $c:\mathbb{R}\longrightarrow E$ be a curve. Then
$c$ is $C^\infty$ if and only if $c$ is $Lip^k$ for all $k\in \N$.
\end{proposition}

The space $C^{\infty}\left(  \mathbb{R},E\right)  $ of such curves does
not depend on the locally convex topology on $E$ but only on its associated
\textit{bornology} (system of bounded sets). Note that the topology can vary
considerably without changing the bornology; the \textit{bornologification}
$E_{\text{born}}$ of $E$ is the finest locally convex structure having the
same bounded sets.

One can note that the link between continuity and smoothness in infinite
dimension is not as tight as in finite dimension: there are smooth maps which
are not continuous!

\bigskip

The $c^{\infty}$-topology on a l.c.v.s. is the final topology with respect to
all smooth curves $\mathbb{R}\rightarrow E$; it is denoted by $c^{\infty}E$.
Its open sets will be called $c^{\infty}$-open.

For every absolutely convex closed bounded set $B$, the linear span $E_{B}$ of
$B$ in $E$ is equipped with the Minkowski functional $p_{B}\left(  v\right)
=\inf\left\{  \lambda>0:v\in\lambda.B\right\}  $ which is a norm on $E_{B}$.

We then have the following characterization of $c^{\infty}$-open sets (\cite{KrMi}
Theorem 2.13):

\begin{proposition}
\label{P_cinfty-open} $U\subset E$ is $c^{\infty}$-open if and only if
$U\cap E_{B}$ is open in $E_{B}$ for all absolutely convex bounded subsets
$B\subset E$.
\end{proposition}

\begin{remark}
The $c^{\infty}$-topology is in general finer than the original topology and
$E$ is not a topological vector space when equipped with the $c^{\infty}$-topology.

For Fr\'{e}chet spaces and so Banach spaces, this topology coincides with the
given locally convex topology.
\end{remark}

\begin{definition}
\label{D_LCVTBornological}A locally convex vector space $E$ is called
bornological if any bounded linear mapping\footnote{A linear map between
locally convex vector spaces is \textit{bounded} if it maps every bounded set
to a bounded one.} $f:E\longrightarrow F$ (where $F$ is any Banach space) is continuous.
\end{definition}

\begin{lemma}
Let $E$ be a bornological vector space. The $c^{\infty}$-topology and
the locally convex topology coincide (i.e. $c^{\infty}E=E$) if the closure of
subsets in $E$ is formed by all limits of sequences in the subset.
\end{lemma}

\begin{definition}
\label{D_SmoothMapBetweenLCTVS}Let $E$ and $F$ be l.c.t.v.s. A mapping
$f:E\longrightarrow F$ is called conveniently  smooth if it maps smooth curves into smooth
curves, i.e. if $f\circ c\in C^{\infty}\left(  \mathbb{R},F\right)  $ for all
$c\in C^{\infty}\left(  \mathbb{R},E\right)  $.
\end{definition}

\medskip
 Note  that in finite dimensional spaces $E$ and $F$ this corresponds
to the usual notion of smooth mappings as proved by Boman (see \cite{Bom}).

\medskip
In finite-dimensional analysis, we use the Cauchy condition, as a necessary
condition for the convergence of a sequence, to define completeness of the
space. In the infinite-dimensional framework, we use the notion of
Mackey-Cauchy sequence (cf. \cite{KrMi}, section 2).

\begin{definition}
\label{D_MackeyCauchySequence} A sequence $\left(  x_{n}\right)  $ in $E$ is
called Mackey-Cauchy if there exists a bounded absolutely convex subset $B$ of $E$ such
that $\left( x_{n} \right)$ is a Cauchy sequence in the normed space $E_B$.
\end{definition}

\begin{definition}
\label{D_ConvenientVectorSpace} A locally convex vector space is said to be $c^{\infty}$-complete or \textit{convenient } if any Mackey-Cauchy sequence converges ($c^{\infty}$-completeness).
\end{definition}

We then have the following characterizations:

\begin{proposition}
A locally convex vector space is \textit{convenient if one of the following
equivalent conditions is satisfied:}
\newline 1. For every absolutely convex closed bounded set $B$ the linear span $E_{B}$ of $B$ in $E$, equipped with
the norm $p_{B}$ is complete.
\newline 2. A curve $c:\mathbb{R}%
\longrightarrow E$ is smooth if and only if $\lambda\circ c$ is smooth for all
$\lambda\in E^{\prime}$ where $E^{\prime}$ is the dual consisting of all
continuous linear functionals on $E$.
\newline 3. Any Lipschitz curve in $E$ is locally Riemann integrable.
\end{proposition}

\begin{example}
\label{Ex_Rinfinity}${}$The vector space $\mathbb{R}^{\infty},$ also denoted
by $\mathbb{R}^{\left(  \mathbb{N}\right)  }$ or $\Phi$, of all finite
sequences is a countable convenient vector space (\cite{KrMi}, 47.1) which is
not metrizable. A basis of $\mathbb{R}^{\infty}$ is $\left(  e_{i}\right)
_{i\in\mathbb{N}^{\ast}}$ where $e_{i}=\left(  0,\dots,0,\underset
{i^{th}\ \text{term}}{1},0,\dots\right)  $.
\end{example}

%For  the \textit{analytic} class of curves in a convenient space we have the following characterization (\cite{KrMi} Corollary 9.4)

%\begin{proposition}
%If  $c:\mathbb{R}\longrightarrow E$ is a curve in a convenient
%vector space $E$ the following properties are equivalent:
%\begin{enumerate}
%\item   $l\circ c:\mathbb{R}\longrightarrow E$  is real analytic for all $l$ in some family of bounded linear functionals which generates the bornology of E.
%\item  $c:\mathbb{R}\longrightarrow \mathbb{R}$ is real analytic for all  bounded linear functionals $l$ on $E$.
%\end{enumerate}
%A curve satisfying these equivalent conditions will be called real analytic.
%\end{proposition}

%Therefore a map  $f:E\longrightarrow F$ between convenient spaces  is called  \textit{  analytic} ($c^{\omega}$ for short) if $f$ is
%smooth and $f\circ c$ is an analytic curve in $F$ for any analytic curve
%$c:\mathbb{R}\rightarrow E$ in $E$.\newline

\bigskip

\begin{theorem}
Let $U$ be a $c^{\infty}$-open set of a convenient vector space $E$ and let
$F$ and $G$ be convenient vector spaces.

1. The space $C^{\infty}\left(  U,F\right)  $ may be endowed with a structure
of convenient vector space. The subspace $L\left(  E,F\right)  $ of all
bounded linear mappings from $E$ to $F$ is closed in $C^{\infty}\left(
E,F\right)$.

2. The category is cartesian closed, i.e. we have the natural diffeomorphism:%
\[
C^{\infty}\left(  E\times F,G\right)  \simeq C^{\infty}\left(  E,C^{\infty
}\left(  F,G\right)  \right).
\]

3. The differential operator%
\[
d   :C^{\infty}\left(  E,F\right)  \longrightarrow C^{\infty}\left(
E,L\left(  E,F\right)  \right)
\]
\[
df\left(  x\right)  v   =\underset{t\longrightarrow0}{\lim}\dfrac{f\left(
x+tv\right)  -f\left(  x\right)  }{t}
\]
\newline exists and is linear and smooth.

4. The chain rule holds:%
\[
d\left(  f\circ g\right)  \left(  x\right)  v=df\left(  g\left(  x\right)
\right)  dg\left(  x\right)  v.
\]
\end{theorem}

\medskip

\begin{proposition}
\label{P_PreservationCompletness} The following constructions preserve $c^{\infty}$-completeness: limits, direct sums, strict direct limits of sequences of closed embeddings.
\end{proposition}

\smallskip
In general, an inductive limit of $c^{\infty}$-complete spaces needs not be $c^{\infty}$-complete (cf. \cite{KrMi}, 2.15, example).

\medskip

According to  \cite{KrMi} section  27.1, a $C^\infty$-{\it atlas} modeled  on a set $M$  modeled on a convenient space $E$ is a family $\{(U_\alpha, u_\alpha) \}_{\alpha\in A}$ of subsets $U_\alpha$ of $M$ and maps $u_\alpha$ from $U_\alpha$ to $E$  such that:

$\bullet$  $u_\alpha$ is a bijection of $U_\alpha$ onto a $c^\infty$-open  subset of  $E$ for all $\alpha\in A$;

$\bullet$  $M=\displaystyle\bigcup_{\alpha\in A}U_\alpha$;

$\bullet$ for any $\alpha$ and $\beta$ such that  $U_{\alpha\beta}=U_\alpha\cap U_\beta\not=\emptyset$, \\
$u_{\alpha\beta}=u_\alpha\circ u_\beta^{-1}:u_\beta(U_{\alpha\beta})\longrightarrow u_\alpha(U_{\alpha\beta})$ is a conveniently  smooth map.\\

Classically, we have a notion of equivalent  $C^\infty$-atlases on  $M$. An equivalent class of $C^\infty$-atlases on $M$ is a maximal  $C^\infty$-atlas.
Such an atlas defines a topology on $M$ which is not in general Hausdorff.

\begin{definition}\label{nnhconvenient} A maximal  $C^\infty$-atlas on $M$  is called a non necessary Hausdorff convenient manifold structure on $M$ (n.n.H. convenient manifold $M$ for short); it is called a Hausdorff convenient manifold structure on $M$  when the topology defined by this  atlas is a Hausdorff topological space.
\end{definition}
Following the classical framework, when $E$ is a Banach space (resp. a Fr\'echet space)  we  say that $M$ is a \textit{Banach manifold} (resp. \textit{Fr\'echet manifold}) if $M$ is provided with a $C^\infty$- atlas (modeled on $E$) which generates a Hausdorff topological space.\\

The notion of vector bundle modeled on a convenient space over a n.n.H. convenient manifold  is defined in a classic way (cf.  (\cite{KrMi}, 29). Note that since a convenient space is Hausdorff, a vector bundle modeled on a convenient space has a natural structure of n.n.H. convenient manifold which is Hausdorff if and only if  the base is a Hausdorff convenient manifold.% When $E$ is a Banach space (resp. a Fr\'echet space)  and $M$ is a  Hausdorff topological space, we  say that $M$ is a \textit{ Banach manifold} (resp. \textit{Fr\'echet manifold}).

\section{\label{*DL_TopologicalVectorSpaces}Direct limits of topological vector spaces}

In this section the reader is referred to \cite{Bou2}, \cite{Glo1} and \cite{Glo2}.

Let $\left(  I,\leq\right)  $ be a directed set. A \textit{direct system} in a
category $\mathbb{A}$ is a pair $\mathcal{S}=\left(  X_{i},\varepsilon_{i}
^{j}\right)  _{i\in I,\ j\in I,\ i\leq j}$ where $X_{i}$ is an object of the
category and $\varepsilon_{i}^{j}:X_{i}\longrightarrow X_{j}$ is a morphism
(\textit{bonding map}) where:

\begin{enumerate}
\item $\varepsilon_{i}^{i}=\operatorname{Id}_{X_{i}}$;

\item $\forall\left(  i,j,k\right)  \in I^{3},i\leq j\leq k \Rightarrow \varepsilon
_{j}^{k}\circ\varepsilon_{i}^{j}=\varepsilon_{i}^{k}$.
\end{enumerate}

A \textit{cone} over $\mathcal{S}$ is a pair $\left(  X,\varepsilon
_{i}\right)  _{i\in I}$ where $X\in\operatorname*{ob}\mathbb{A}$ and
$\varepsilon_{i}:X_{i} \longrightarrow X$ is such that $\varepsilon_{j}%
\circ\varepsilon_{i} ^{j}=\varepsilon_{i}$ whenever $i\leq j$.

A cone $\left(  X,\varepsilon_{i}\right)  _{i\in I}$ is a \textit{direct
limit} of $\mathcal{S}$ if for every cone $\left(  Y,\theta_{i}\right)  _{i\in
I}$ over $\mathcal{S}$ there exists a unique morphism $\psi:X\longrightarrow
Y$ such that $\psi\circ\varepsilon_{i}=\theta_{i}$. We then write
$X=\underrightarrow{\lim}\mathcal{S}$ or $X=\underrightarrow{\lim}X_{i}$.

When $I=\mathbb{N}$ with the usual order relation, countable direct systems
are called \textit{direct sequences}.

\subsection{Direct limit of sets}

Let $\mathcal{S}=\left(  X_{i},\varepsilon_{i}^{j}\right)  _{i\in I,\ j\in
I,\ i\leq j}$ be a direct system of sets (we then have $\mathbb{A}
=\mathbb{SET}$).

Let $\mathcal{U}=\coprod\limits_{i\in I}X_{i}=\left\{  \left(  x,i\right)
:x\in X_{i}\right\}  $ be the disjoint union of the sets $X_{i}$ with the
canonical inclusion%

\[%
\begin{array}
[c]{cccc}%
\boldsymbol{\imath}_{i}: & X_{i} & \longrightarrow & \mathcal{U}\\
& x & \mapsto & \left(  x,i\right)
\end{array}
.
\]

We define an equivalence relation on $\mathcal{U}$ as follows:
$\boldsymbol{\imath}_{i}\left(  x\right)  \sim\boldsymbol{\imath}_{j}\left(
y\right)  $ if there exists $k\in I$: $i\leq k$ and $j\leq k$ s.t.
$\varepsilon_{i}^{k}\left(  x\right)  =\varepsilon_{j}^{k}\left(  y\right)  $.
We then have the quotient set $X=\mathcal{U}/\sim$ and the map $\varepsilon
_{i}=\pi\circ\boldsymbol{\imath}_{i}$ where $\pi:\mathcal{U}\longrightarrow
\mathcal{U}/\sim$ is the canonical quotient map.

Then $\left(  X,\varepsilon_{i}\right)  $ is the direct limit of $\mathcal{S}$
in the category $\mathbb{SET}$.

If each $\varepsilon_{i}^{j}$ is injective then so is $\varepsilon_{i}$, whence
$\mathcal{S}$ is equivalent to the direct system of the subsets $\varepsilon
_{i}\left(  X_{i}\right)  \subset X$, together with the inclusion maps.

\subsection{Direct limit of topological spaces}

If $\mathcal{S}=\left(  X_{i},\varepsilon_{i}^{j}\right)  _{i\in I,\ j\in
I,\ i\leq j}$ is a direct system of topological spaces and continuous maps,
then the direct limit $\left(  X,\varepsilon_{i}\right)  _{i\in I}$ of the
sets becomes the direct limit in the category $\mathbb{TOP}$ of topological
spaces if $X$ is endowed with the direct limit topology ($\mathit{DL}$-\textit{topology} for short), i.e. the finest topology which makes the maps $\varepsilon_{i}$
continuous. So $O\subset X$ is open if and only if $\varepsilon_{i}
^{-1}\left(  O\right)  $ is open in $X_{i}$ for each $i\in I$.
\smallskip

When $\mathcal{S}=\left(  X_{n},\varepsilon_{n}^{m}\right)  _{n\in
\mathbb{N}^{\ast},\ m\in\mathbb{N}^{\ast},\ n\leq m}$ is a direct sequence of
topological spaces such that each $\varepsilon_{n}^{m}$ is injective, without
loss of generality, we may assume that we have
\[
X_{1}\subset X_{2}\subset\cdots\subset X_n\subset X_{n+1}\subset \cdots
\]
and $\varepsilon_{n}^{n+1}$ becomes the natural inclusion. Therefore
$\mathcal{S}$ will be called \textit{an ascending sequence of topological
spaces} and simply denoted $(X_{n})_{n\in\mathbb{N}^{\ast}}.$

Moreover, if each $\varepsilon_{n}^{m}$ is a topological embedding, then we will
say that $\mathcal{S}$ is a \textit{strict ascending sequence of topological spaces} (\textit{expanding sequence }in the terminology of \cite{Han}).
In this situation, each $\varepsilon_{n}$ is a topological embedding on $X_{n}$ in $X=\underrightarrow{\lim}X_{n}$.

\smallskip
Let us give some properties of ascending sequences of topological spaces (\cite{Glo2}, Lemma 1.7):

\begin{proposition}
\label{P_TopologicalPropertiesAscendingSequenceTopologicalSpaces}Let
$(X_{n})_{n\in\mathbb{N}^{\ast}}$ be an ascending sequence of topological
spaces. Equip $X=\bigcup\limits_{n\in\mathbb{N}^{\ast}}X_{n}$ with the final
topology with respect to the inclusion maps $\varepsilon_{n}:X_{n}%
\longrightarrow X$ (i.e. the $DL$-topology). Then we have:

1. If each $X_{n}$ is $T_{1}$, then $X$ is $T_{1}$.

%2. If each $X_n$ is $T_4$ then $X$ is $T_4$

2. If $O_{n}\subset X_{n}$ is open and $O_{1}\subset O_{2}\subset\cdots$, then
$O=\bigcup\limits_{n\in\mathbb{N}^{\ast}}O_{n}$ is open in $X$ and the
$DL$-topology on $O=\underrightarrow{\lim}O_{n}$ coincides with the topology
induced by $X$.

3. If each $X_{n}$ is locally compact, then $X$ is Hausdorff.

4. If each $X_{n}$ is $T_{1}$ and $K\subset X$ is compact, then $K\subset
X_{n}$ for some $n$.
\end{proposition}

Unfortunately, in general, a direct limit  of Hausdorff topological spaces is not Hausdorff (see \cite{Her} for an example of such a situation). Sufficient conditions on $\left( X_n \right)_{n\in\mathbb{N}^*}$  under which the direct limit  $X=\bigcup\limits_{n\in\mathbb{N}^{\ast}}X_{n}$ is Hausdorff can be found in \cite{HaSt}. However, we have:

\begin{proposition}\label{directlimitT4}
Let $(X_{n})_{n\in\mathbb{N}^{\ast}}$ be a strict ascending sequence of topological spaces; equip $X=\bigcup\limits_{n\in\mathbb{N}^{\ast}}X_{n}$ with  the $DL$-topology. Then we have:
\begin{enumerate}
\item Assume that  for each $n$, $X_{n}$ is closed in $X_{n+1}$.

(a) If each $X_n$ is normal then $X$ is normal.

(b) If each $X_n$ is Hausdorff and paracompact then $X$ is normal.\\
In particular, in each previous situation, $X$ is Hausdorff.
\item  Assume that for each $n$,  $X_{n}$ is open in $X_{n+1}$ and Hausdorff, then $X$ is Hausdorff.
\end{enumerate}
\end{proposition}

\begin{proof} 1.(a) see \cite{Han} Proposition 4.3. (i).\\
1.(b)  It is  well known that any Hausdorff and paracompact topological space is normal and then such a topological space is Hausdorff.\\
 For part  2.  see  \cite{Han} Proposition 4.2.
\end{proof}

\begin{remark}
\label{ngeqn0} \label{top ascngeqn0} The direct limit of an ascending sequence
$(X_{n})_{n\in\mathbb{N}^{\ast}}$  is equal to the direct
limit of $(X_{n})_{n\in\mathbb{N}^{\ast},\ n\geq n_{0}}$ %in the category $\mathbb{SET}$.
\end{remark}

Let $(X_{n},i_{n}^{m})_{n\leq m,\ m\in\mathbb{N}^{\ast},n\in\mathbb{N}^{\ast}%
}$ and $(Y_{n},j_{n}^{m})_{n\leq m,\ m\in\mathbb{N}^{\ast},n\in\mathbb{N}%
^{\ast}}$ be two ascending sequences of topological spaces. Then assume that
we are given a sequence of maps $f_{n}:X_{n}\longrightarrow Y_{n}$ which is
\textit{consistent}, i.e. that we have for any $n\leq m$, $f_{m}\circ
i_{n}^{m}=j_{n}^{m}\circ f_{n}$. Then these sequences induce a map
$f:X=\underrightarrow{\lim}X_{n}\longrightarrow Y=\underrightarrow{\lim}Y_{n}$
s.t. $f\circ i_{n}=i_{n}\circ f_{n}$ where $i_{n}:X_{n}\longrightarrow
X=\underrightarrow{\lim}X_{n}$ and $j_{n}:Y_{n}\longrightarrow
Y=\underrightarrow{\lim}Y_{n}$ are the associate inclusions respectively.

If every map $f_{n}$ is continuous, then the induced map $f$ is continuous with
respect to the $DL$-topologies on $X$ and $Y$ (continuity criterion,
\cite{HSTH}).

\subsection{\label{*DL_BanachSpaces}Direct limit of Banach spaces}

Let $(E_{n})_{n\in\mathbb{N}^{\ast}}$ be an ascending sequence of Banach
spaces. It is easy to see that we can choose a norm $||\;||_{n}$ on $E_{n}$,
for $n\in\mathbb{N}^{\ast}$, such that:
\[
||\;||_{n+1}\leq||\;||_{n}\text{ on }E_{n}\text{ for each }n\in\mathbb{N}%
^{\ast}.
\]

\medskip

\textit{In this paper, we always make such a choice}.

\medskip

Given such an ascending sequence of Banach spaces, then $E=\bigcup\limits_{n\in
\mathbb{N}^{\ast}}E_{n}$ is called the \textit{direct limit } of this sequence. The finest locally convex vector topology making each inclusion map
$E_{n}\longrightarrow E$ continuous, is called the \textit{locally convex
direct limit topology} and denoted $LCDL$-topology for short.

A convex set $O\subset E$ is open in this topology if and only if $O\cap
E_{n}$ is open in $E_{n}$ for each $n\in\mathbb{N}^{\ast}$.

If $E_{n}$ is a Banach subspace of $E_{n+1}$ for each $n\in\mathbb{N}^{\ast}$, we have a \textit{strict ascending sequence of Banach spaces}.

\begin{definition}
\label{D_LBspace}A locally convex limit of ascending sequence of Banach spaces is called an (LB)-space.

If the sequence is strict we speak of LB-space or strict (LB)-space.
\end{definition}

%\begin{definition}
%\label{B_SilvaSpace}A locally convex limit of ascending sequence of Banach
%spaces where each inclusion is compact is called a Silva space or a (DFS)-space (for dual of Fr\'{e}chet Schwartz space).
%\end{definition}

Of course, an (LB)-space does not have a structure of convenient space in general.
%We mention some important cases for which an (LB)-space has a structure of convenient space:\newline
However, since every Banach space is convenient, then each LB-space  has a structure of convenient space  (see \cite{KrMi} Theorem 2.15). In particular the direct limit of an ascending sequence of finite dimensional Banach spaces has a structure of convenient space.

%$(b)$ Silva space (cf. \cite{FrKr} 4.4.39);\newline

%$()$ each Banach space $E_{n}$ is finite dimensional (which corresponds to
%the intersection of the two previous cases).\newline

%\medskip

%Note that, in each of these previous cases, the $LCDL$-topology on $E$ induces
%the given Banach topology on each member $E_{n}$ and the $c^{\infty}$-topology
%is finer than the $LCDL$-topology.

%\medskip

%\noindent\textit{In the whole paper, we consider the following type of (LB)-space}:

%\begin{definition}
%\label{D_LBCspace} We say that an (LB)-space $E=\underrightarrow{\lim}E_{n}$ is
%an (LBC)-space if $E$ is endowed with a convenient space structure.
%\end{definition}

%\medskip

Note that if $E=\underrightarrow{\lim}E_{n}$ and $F=\underrightarrow{\lim
}F_{n}$ are LB-spaces, then we can identify $E\times F=\underrightarrow
{\lim}(E_{n}\times F_{n})$ with $\underrightarrow{\lim}E_{n}\times
\underrightarrow{\lim}E_{n}$ as locally convex topological spaces (cf.
\cite{HSTH}, Theorem 4.3) and so $E\times F$ is a convenient space (cf.
\cite{KrMi}). \\%Therefore, if $(E_{n})_{n\in\mathbb{N}^{\ast}}$ (resp.
%$(F_{n})_{n\in\mathbb{N}^{\ast}}$) is an ascending sequence of type (a) (resp.
%(b)) then $(E_{n}\times F_{n})_{n\in\mathbb{N}^{\ast}}$ is neither of type
%$(a)$ nor of type $(b)$ but $E\times F$ is an (LBC)-space.\newline
\medskip
We now consider a general situation which gives rise to a convenient structure on a direct limit of an ascending sequence of Banach spaces.

\begin{proposition}
\label{P_Mixtes_a-b} Let $(E_{n})_{n\in\mathbb{N}^{\ast}}$ be an ascending sequence of
Banach spaces. Assume that there exists an infinite subset
$I\subset\mathbb{N}^{\ast}$ such that $E_{I}=\bigcup\limits_{i\in I}%
E_{i}=\underrightarrow{\lim}E_{i}$ is an LB-space.
%one of the following properties is satisfied for each $n\in\N^\ast$
%(i) the inclusion $\iota_n^{n+1}:E_nlongrightarrow E_{n+1}$ is compact;
%(ii) the topology on $E_n$ is induced by the topology of $E_{n+1}$.
Then $E=\bigcup\limits_{n\in\mathbb{N}^{\ast}}E_{n}=E_{I}$ and $E$ is
an LB-space.
\end{proposition}

\begin{proof}
%at first note that if there exists $n_0\in \N$ such that one and only one the properties (i) and (ii) is satisfied for $n\geq n_0$,
%we are in one of the context (a) or (b) for $n\geq n_0$.From Remark \ref{top ascngeqn0} the proof is ended. \\
%Assume that there exists a subset $I\subset\N$ such that (i) is satisfied for any $n\in I$.
We set $J=\mathbb{N}^{\ast}\setminus I$. An index of $I$ (resp. $J$) will be
denoted $i_{l}$, $l\in\mathbb{N}^{\ast}$ (resp. $j_{k}$, $k\in\mathbb{N}%
^{\ast}$).
%Therefore $(E_{i_l})_{l\n \N}$ (resp. $(E_{j_k})_{k_\n \N}$ is a sequence of ascending Banach space of context (a) (resp. context (b)).
In the category of $\mathbb{SET}$, we have $E_{I}=\bigcup\limits_{l\in
\mathbb{N}^{\ast}}E_{i_{l}}=\underrightarrow{\lim}E_{i_{l}}$ and
$E_{J}=\bigcup\limits_{k\in\mathbb{N}^{\ast}}E_{j_{k}}=\underrightarrow{\lim
}E_{j_{k}}$. From our assumption, $E_{I}$ is a convenient space and for any
$i_{l}\in I$ there exists $j_{k}\in J$ such that $E_{i_{l}}\subset E_{j_{k}}$
and conversely. Therefore we have the equality $E_{I}=E_{j}=E$. Thus in the
category $\mathbb{SET}$ we have
\[
E=\underrightarrow{\lim}E_{i_{l}}=\underrightarrow{\lim}E_{j_{k}}.
\]
Consider an open set $O=\bigcup\limits_{l\in\mathbb{N}^{\ast}}O_{i_{l}}$ of
$E_{i}$. Given any $j_{k}\in J$, the space $E_{j_{k}}$ is contained in some
$E_{i_{l_{0}}}$ for $i_{l_{0}}<j_{k}$. Therefore $O\cap E_{j_{k}}=O\cap
O_{i_{l_{0}}}$. As the inclusion $\iota_{i_{k}}^{i_{l_{0}}}:E_{j_{k}
}\longrightarrow E_{i_{l_{0}}}$ is continuous (as composition of a finite
number of continuous inclusions), $O\cap E_{j_{k}}$ is an open set of
$E_{j_{k}}$. It follows that the $DL$-topology of $E=\underrightarrow{\lim
}E_{n}$ and $E_{I}$ coincide. In particular $E$ is Hausdorff.
%The same argument works also for $E_J$.

Of course, the algebraic structure of vector space on each set $E$, $E_{i}$
coincide and then $E$ and $E_{I}$ have the same convex sets. As a set, $O$ is
an open set of the $LCDL$-topology on $E$ if and only if $O$ is convex and
$O\cap E_{n}$ is open in $E_{n}$ for all $n\in\mathbb{N}^{\ast}$. Again, as
each $E_{j_{k}}$ is contained is some $E_{i_{l_{0}}}$ for some $i_{l_{0}%
}>j_{k}$, it follows that $O$ is an open set of the $LCDL$-topology on $E$ if
and only if $O$ is an open set for $LCDL$-topology on $E_{I}$. It follows that
the $LCDL$-topology on $E$ and $E_{I}$ also coincide. Finally, the locally
convex vector spaces $E_{I}$ and $E$ have the same convex bounded sets. It
follows that $E$ is a convenient space with the same structure as $E_{I}$ and
also the same $c^{\infty}$-topology.
\end{proof}

\begin{lemma}
\label{L_LBCspace}${}$ If $E=\underrightarrow{\lim}E_{n}$ is an LB-space,
then for the $LCDL$-topology we have:

\begin{enumerate}
\item[(i)] $E$ is Hausdorff and bounded regular (i.e. every bounded
subset of $E$ is contained in some $E_{n}$). 

\item[(ii)] Let $f:E\longrightarrow F$ be a linear map where $F$ is a Banach
space.\newline 
The following properties are equivalent:

$(1)$ $f$ is bounded;

$(2)$ each restriction $f_{n}$ of $f$ to $E_{n}$ is continuous;

$(3)$ $f$ is continuous.
\end{enumerate}
\end{lemma}

\begin{proof}
(i) As $E$ is endowed with a convenient structure, it must be Hausdorff. On the other hand,   since $E$ is an LB-space,
%$c^\infty$-complete for the $LCDL$-topology. Now from \cite{Flo} 5.6 and 4.3, 
it must be
bounded regular (see for example \cite{PeBo}).
% From \cite{Wen}, the bounded regularity and compact regularity are equivalent properties in this context. 

(ii) $(1)\Longrightarrow(2)$: If $f$ is bounded, then its restriction $f_{n}$ of $f$ to $E_{n}$ is bounded, so $f_{n}:E_{n}\longrightarrow F$ is continuous. \newline
$(2)\Longrightarrow(3)$: Assume that each restriction $f_{n}$ of $f$ to
$E_{n}$ is continuous. To prove that $f$ is continuous, it is sufficient to
show that for any ball $B(0,r)$ in $F$, then $f^{-1}(B(0,r))$ is an open convex
set of $E$. But we have
\[
f^{-1}(B(0,r))=\bigcup\limits_{n\in\mathbb{N}^{\ast}}f_{n}^{-1}(B(0,r)).
\]
As each $f_{n}$ is continuous, $U_{n}=f_{n}^{-1}(B(0,r))$ is an open set of
$E_{n}$. Moreover, $U_{n}\subset U_{n+1}$. On the other hand, as $B(0,r)$ is
convex, so is $U_{n}$. There $U=f^{-1}(B(0,r))=\bigcup\limits_{n\in
\mathbb{N}^{\ast}}U_{n}$ is also convex. So $U$ is an open set of $E$
(relative to the $LCDL$-topology).
\newline
$(3)\Longrightarrow(1)$: Assume now
that $f$ is continuous and consider a bounded set $B$ of $E$. From part (i),
$B$ is contained in some $E_{n}$. But as the inclusion of $E_{n}$ in $E$ is
continuous, $f_{n}$ is continuous and so $f(B)=f_{n}(B)$ is bounded in
$F$.
\end{proof}

\begin{proposition}
\label{P_LBCctoplctop}${}$ On an LB-space $E=\underrightarrow{\lim}E_{n}$,
the $DL$-topology coincides with the $c^{\infty}$-topology.
\end{proposition}

\begin{proof}
Let $B$ be an absolutely convex bounded set of $E$. From Lemma
\ref{L_LBCspace}, $B$ is contained in some $E_{n}$ and then $E_{B}$ is a
vector subspace of $E_{n}$ equipped with the Minkowski norm $p_{B}$. The
closed unit ball $B_{n}$ in $E_{n}$ is also an absolutely convex bounded set
and there exists $\alpha\geq1$ such that $B\subset\alpha.B_{n}$. Remark that
if $p_{n}$ is the given norm on $E_{n}$, then $p_{n}$ is the Minkowski
functional associated to $B_{n} $.
%Now, on one hand, if $v\in B$ then $v$ belongs to $\alpha.B$ so we have
%\[
%p_{B}\leq\alpha p_{n}
%\]
%On the other hand, for
If $v\in E_{B}$, then $\dfrac{v}{p_{B}\left(  v\right)  }$ belongs to
$\alpha.B_{n}$; so we get
\[
p_{n}\leq\alpha p_{B}.
\]
Therefore the inclusion of $E_{B}$ in $E_{n}$ is continuous.
%$p_{B}$ and ${p_{n}}_{|B}$ are two equivalent norms on $E_{B}$. So
%the topology induced by $p_{n}$ on $E_{B}$ is equivalent to the topology on
%$E_{B}$ given by $p_{B}$.\newline
Let $U$ be an open set of the $DL$-topology on $E$. Then for any
$n\in\mathbb{N}^{\ast}$, $U\cap E_{n}$ is open in $E_{n}$. Thus, given any
absolutely convex bounded set $B$ of $E$, if $E_{B}$ is contained in $E_{n}$,
then $U\cap E_{B}=U\cap E_{n}\cap E_{B}$ is open in $E_{B}$. It follows from
Proposition \ref{P_cinfty-open} that $U$ is a $c^{\infty}$-open.\newline
Conversely, if $U$ is a $c^{\infty}$-open, as $E_{n}=E_{B_{n}}$ and the norm
$p_{B_{n}}$ is the given norm on $E_{n}$, again from Proposition
\ref{P_cinfty-open}, it follows that $U$ is an open set of the $DL$-topology
on $E$.
\end{proof}

\medskip

\begin{remark} In Proposition \ref{P_LBCctoplctop}, the fact that an LB-space is regular is essential. More generally, if an (LB)-space is convenient and bounded regular, such a result is also true. However, for the sake of simplicity, we limit ourselves to the LB-space context.
\end{remark}

Recall that from the classical differential calculus in locally convex
topological spaces, for $r\in\mathbb{N}\cup\left\{  {\infty}\right\}  $, the
map $f$ is of class $C^{r}$ ($C^{r}$-map for short) if it is continuous and,
for all $k\in\mathbb{N}$ such that $k\leq r$, the iterated directional
derivatives $d_{k}f(x,y_{1},\cdots,y_{k}):=D_{y_{1}}\cdots D_{y_{k}}f(x)$
exist for all $x\in U$ and $y_{1},\cdots,y_{k}\in E$ and the associated map
$d_{k}f:U\times E^{k}\longrightarrow F$ is continuous. When $r=\infty$ we say
that $f$ is smooth.\newline

We have the following link between $C^{\infty}$-smoothness on each member $E_{n}$ and
conveniently smoothness on $E$ for an LB-space (cf. \cite{Glo2} Lemma 1.9):

%differentiability (\cite{Glo2}, Lemma 1.9) for a map $f:O\longrightarrow F$
%where $O=\bigcup\limits_{n\in\mathbb{N}^{\ast}}O_{n}$ is an ascending sequence
%of open sets ($O_{n}\subset E_{n}$) and $F$ a real topological vector space
%locally convex and Mackey complete: $f$ is $C^{\infty}$ if and only if $f$ is
%$c^{\infty}$.

\begin{lemma}
\label{L_Cinftycinfty}${}$Let $E=\underrightarrow{\lim}E_{n}$ be an LB-space, 
$U\subset E=\underrightarrow{\lim}E_{n}$ an open set of $E$ (for the
$DL$-topology) and $U_{n}=U\cap E_{n}$ the associated open set in $E_{n}$.
Given a map $f:U\longrightarrow F$ where $F$ is a convenient space, $f$
is conveniently smooth  if and only if $f_{n}=f_{|U\cap E_{n}}$ is
$C^{\infty}$  for each $n \in \mathbb{N}^\ast$.
\end{lemma}

\begin{proof}
This proof is an adaption of the proof of Lemma 1.9 of \cite{Glo2}.\newline
First, note that on an open set of a Banach space, we have equivalence between
conveniently smoothness and $C^{\infty}$-differentiability (cf. \cite{Bom}). \newline
Assume that $f$ is conveniently smooth on $U$. Given any smooth curve $\gamma:\mathbb{R}\longrightarrow U\cap E_{n}$, then $\gamma$ is a smooth curve in $U$,
so $f\circ\gamma$ is smooth. As $U\cap E_{n}$ is an open set of a Banach space, it
follows that $f_{n}=f_{|U\cap E_{n}}$ is $C^{\infty}$. \newline

Conversely \footnote{We are grateful to the anonymous referee for this part of the proof.}  assume that $f_{n}=f_{|U\cap E_{n}}$ is $C^{\infty}$ and  let $\gamma
:\mathbb{R}\longrightarrow U$ be a smooth curve.
Fix  $ a < b$ and  $k \in  \N$. From the bounded regularity of $E$, it follows that there exists $N\in \N$,  $N\geq n$, such that all the  sets
$$\{\displaystyle \frac{(\gamma^{(j)}(t) - \gamma^{(j)}(s))}{s-t}: (s,t) \in ]a, b[\times]a, b[, s\not= t \} \;\;\textrm{ and } \gamma^{(j)}(]a, b[)$$
where $\gamma^{(j)}:\R\longrightarrow E$ is the $j^{th}$ derivative of $\gamma$, are contained and bounded in $E_N$
for all $j=0,\dots,k$. Since  $\gamma^{(j)}_{|]a,b]}$ is a Lipschitz curve in the Banach space $E_N$, so there exists a primitive
 $\eta_j$ on $]a,b[$. Thus $\eta_j$ and  $\gamma^{(j-1)}_{|]a,b]}$ have the same derivative $\gamma^{(j)}_{|]a,b]}$ and so  these curves  differ by a constant for all $j=1,\dots,k$.
 We conclude that $\gamma^{(j-1)}:]a,b[\longrightarrow E_N$ is a $C^1$-curve whose derivative  is $\gamma^{(j)}_{|]a,b]}$. As a consequence, $\gamma_{|]a,b]}$ is a $Lip^k$-curve in $E_N$.
 We can choose $b-a$ small enough  such that $\gamma(]a,b[)\subset U\cap E_n\subset U\cap E_N$. Since  the restriction of $f$ to $U\cap E_n$ is smooth, it follows that, for any $k\in \N$, $f\circ \gamma:]a,b[\longrightarrow \R$ is $Lip^k$ for any $a<b$ where $b-a$ is small enough. From Proposition \ref{lipschitz-smooth}, we get that  $f\circ \gamma$ is smooth in $E$ and so $f$ is  conveniently smooth (cf. Definition \ref{D_SmoothMapBetweenLCTVS}).
  \end{proof}

\begin{proposition}
\label{P_DirectLimitMap}${}$Let $E=\underrightarrow{\lim}E_{n}$ and
$F=\underrightarrow{\lim}F_{n}$ be LB-spaces and $U_{1}\subset\cdots\subset
U_{n}\subset\cdots$ an ascending sequence $(U_{n})$ of open sets of $(E_{n})$
and set $U=\bigcup\limits_{n\in\mathbb{N}^{\ast}}U_{n}$. Assume that we have
a sequence of $f_{n}:U_{n}=U\cap E_{n}\longrightarrow F_{n}$ which are
$C^{\infty}$.

\begin{enumerate}
\item[(i)] Then $f=\underrightarrow{\lim}f_{n}$ is a conveniently smooth map from $U$ to
$F$.
\item[(ii)] Let $f_{n}:E_{n}\longrightarrow F_{n}$ be a sequence of continuous
linear maps; then $f$ is a linear map from $E$ to $F$ which is conveniently smooth and continuous for the $DL$-topologies.
\end{enumerate}
\end{proposition}

\begin{proof}
(i) According to Proposition
\ref{P_TopologicalPropertiesAscendingSequenceTopologicalSpaces}, $U$ is an open
set of $E$ and $U=\underrightarrow{\lim}U_{n}$. So $f=\underrightarrow{\lim
}f_{n}:U\longrightarrow F$ is a well defined continuous map (for the
$DL$-topology). From Lemma \ref{L_Cinftycinfty}, $f$ is then a conveniently smooth map
from $U$ to $F$.\newline

Under the assumption of (ii), the associated map $f=\underrightarrow{\lim
}f_{n}$ is linear and continuous. Now, as $f_{n}$ is continuous linear between
Banach spaces, $f_{n}$ must be $C^{\infty}$ and by part (i), $f$ must be
conveniently smooth.
\end{proof}

\section{\label{*DL_Manifolds}Direct limit of manifolds}

\subsection{Direct limit of ascending sequence of Banach manifolds}

Let $M$ be a n.n.H. convenient manifold modeled on a convenient space $E$ and $TM$
its kinematic tangent bundle (cf. \cite{KrMi}, 28.1).
\newline
We adapt  to our context the notion of weak submanifold used in \cite{Pel}.

\begin{definition}
A \textit{weak submanifold} of $M$ is a pair $(N,\varphi)$ where $N$ is a
n.n.H. convenient connected  manifold (modeled on a convenient space $F$) and
$\varphi:N\longrightarrow M$ is a conveniently smooth map such that:

--- there exists a continuous injective linear map $i:F\longrightarrow E$ (for
the structure of l.c.v.s. of $E$)

--- $\varphi$ is an injective conveniently smooth map and the tangent map
$T_{x}\varphi:T_{x}N\longrightarrow T_{\varphi(x)}M$ is an injective
continuous linear map with closed range  for all $x\in N$.
\end{definition}

Note that for a weak submanifold $\varphi:N\longrightarrow M$, on the subset
$\varphi(N)$ of $M$, we have two topologies:

--- the induced topology from $M$;

--- the topology for which $\varphi$ is a homeomorphism from $N$ to
$\varphi(N)$.
\newline
With this last topology, via $\varphi$, we get on $\varphi(N)$ a n.n.H.
convenient manifold  structure modeled on $F$. Moreover, the inclusion from
$\varphi(N)$ into $M$ is continuous as a map from the  manifold
$\varphi(N)$ to $M$. In particular, if $U$ is an open set of $M$, then
$\varphi(N)\cap U$ is an open set for the topology of the  manifold on
$\varphi(N)$. Therefore, if $M$ is Hausdorff so is $\varphi(N)$. \newline

\begin{lemma}
\label{L_asEn} Let $\mathcal{M}=(M_{n})_{n\in\mathbb{N}^{\ast}}$ be an
ascending sequence of Banach $C^{\infty}$-manifolds, where $M_n$ is modeled on the Banach
space $E_{n}$ and where the inclusion $\varepsilon_{n}^{n+1}:M_{n}\longrightarrow
M_{n+1}$ is a $C^{\infty}$ injective map such that $\varepsilon
_{n}^{n+1}(M_{n})$ is a weak submanifold of $M_{n+1}$.

\begin{enumerate}
\item[(i)] There exist injective continuous linear maps $\iota_{n}%
^{n+1}:E_{n}\longrightarrow E_{n+1}$ such that $(E_{n})_{n\in\mathbb{N}}$ is
an ascending sequence of Banach spaces.

\item[(ii)] Assume that for $x\in M=\underrightarrow{\lim}M_{n}$, there exists
a family of charts $(U_{n},\phi_{n})$ of $M_{n}$, for each $n\in
\mathbb{N}^{\ast}$, such that:

-- $(U_{n})_{n\in\mathbb{N}^{\ast}}$ is an ascending sequence of chart domains;

-- $\phi_{n+1}\circ\varepsilon_{n}^{n+1}=\iota_{n}^{n+1}\circ\phi_{n}$.

Then $U=\underrightarrow{\lim}U_{n}$ is an open set of $M$ endowed with the $DL$-topology and $\phi
=\underrightarrow{\lim}\phi_{n}$ is a well defined map from $U$ to $E=\underrightarrow{\lim}E_{n}$.
Moreover, $\phi$ is a continuous homeomorphism from $U$ onto the open set
$\phi(U)$ of $E$.
\end{enumerate}
\end{lemma}

%\begin{remark}
%\label{R_ngeqn0}
Note that, from  Remark \ref{ngeqn0}, the direct limit of $\mathcal{M}%
=(M_{n})_{n\in\mathbb{N}^{\ast}}$ is the same as the direct limit of
$(M_{n})_{n\in\mathbb{N}^{\ast},n\geq n_{0}}$. The result of part (ii) of this
Lemma is still true if there exists an integer $n_{0}$ such that the
assumptions of part (ii) are satisfied for all $n\geq n_{0}$.
%\end{remark}

\begin{proof}
(i) As $(M_{n},\varepsilon_{n}^{n+1})$ is a weak submanifold of $M_{n}$, there
exists an injective continuous linear map $i_{n}^{n+1}:E_{n}\longrightarrow
E_{n+1}$ for each $n$. Therefore $(E_{n})_{n\in\mathbb{N}^{\ast}}$ is an
ascending sequence of Banach spaces.\newline

(ii) Under the assumption of part (ii), we set $V_{n}=\varphi_{n}(U_{n})$. First, from Proposition
\ref{P_TopologicalPropertiesAscendingSequenceTopologicalSpaces}, as $V_{n}$ is
an open of $E$, we have $U=\bigcup\limits_{n\in\mathbb{N}^{\ast}}%
U_{n}=\underrightarrow{\lim}U_{n}$ and $V=\bigcup\limits_{n\in\mathbb{N}%
^{\ast}}V_{n}=\underrightarrow{\lim}V_{n}$. Moreover, $U$ (resp $V$) is an open
neighborhood of $x$ (resp. $y$). According to the continuity criterion,
$f=\underrightarrow{\lim}f_{n}$ is a continuous map from $U$ to $V$ which is
injective and surjective. As each $f_{n}$ is a homeomorphism, we can apply the
same arguments to the family $f_{n}^{-1}$, which ends the proof.
\end{proof}

\begin{definition}
\label{DLchart} We say that an ascending sequence $\mathcal{M}=(M_{n}%
)_{n\in\mathbb{N}^{\ast}}$ of Banach $C^{\infty}$-manifolds has the direct
limit chart property at $x\in M=\underrightarrow{\lim}M_{n}$ if $(M_{n}%
)_{n\in\mathbb{N}^{\ast}}$ satisfies the assumptions of Lemma \ref{L_asEn}
part (ii).
\end{definition}

Once more, note that the direct limit of $\mathcal{M}=(M_{n})_{n\in
\mathbb{N}^{\ast}}$ is the same as the direct limit of $(M_{n})_{n\in
\mathbb{N}^{\ast},n\geq n_{0}}$ (cf. Remark \ref{ngeqn0}).

\begin{example}
\label{Ex_existDLChart}${}$The existence of a direct limit chart is a natural
requirement which is satisfied in many examples. We give some of them below.

1. According to Theorem 3.1 of \cite{Glo2}, if $(M_{n})_{n\in\mathbb{N}^{\ast}%
}$ is an ascending sequence of $C^{\infty}$ finite dimensional manifolds, then
such a sequence has the direct limit chart property at any $x\in M$.

2. If $M$ is a compact analytic manifold, it is well known that the set
$\operatorname{Diff}(M)$ of analytic diffeomorphisms of $M$ can be described
as a direct limit of an ascending sequence of Banach manifolds $(M_{n}%
)_{n\in\mathbb{N}^{\ast}}$ which has the direct limit chart property for any
point of $\operatorname{Diff}(M)$. Note that, in this case,
$(M_{n})_{n\in\mathbb{N}^{\ast}}$ is modeled on a sequence $(E_{n}%
)_{n\in\mathbb{N}^{\ast}}$ of Banach spaces whose direct limit $E$ is a Silva space
\footnote{A locally convex limit of ascending sequence of Banach spaces where each inclusion is compact is called a Silva space or a (DFS)-space (for dual of Fr\'{e}chet Schwartz space).}.

3. In \cite{Dah}, the reader can find examples of Lie groups which can be
described as direct limits of ascending sequences of Banach manifolds
$(M_{n})_{n\in\mathbb{N}^{\ast}}$ modeled on sequences of Banach spaces
$l^{p}$ whose direct limits $E$ are not Silva spaces.

4. In the introduction of \cite{Glo3}, one can also find many examples of Lie
groups which have the direct limit chart property at each point.

5. Let $(M_{n})_{n\in\mathbb{N}^{\ast}}$ and $(N_{n})_{n\in\mathbb{N}^{\ast}}$
be two ascending sequences of Banach manifolds which have the direct limit
chart property at $x\in M=\underrightarrow{\lim}M_{n}$ and at $y\in
N=\underrightarrow{\lim}N_{n}$ respectively. Then $(M_{n}\times N_{n}%
)_{n\in\mathbb{N}^{\ast}}$ has the direct limit chart property at $(x,y)\in
M\times N=\underrightarrow{\lim}(M_{n}\times N_{n})$. Therefore, given any
Banach manifold $M$ and any ascending sequence $(N_{n})_{n\in\mathbb{N}^{\ast
}}$ of $C^{\infty}$ finite dimensional manifolds, $(M\times N_{n}%
)_{n\in\mathbb{N}^{\ast}}$ has the direct limit chart property at any point of
$M\times N=\underrightarrow{\lim}(M\times N_{n})$.
\end{example}

We now give a general context under which an ascending sequence $\mathcal{M}%
=(M_{n})_{n\in\mathbb{N}^{\ast}}$ of Banach $C^{\infty}$-manifolds has the
direct limit chart property at each point of $M=\underrightarrow{\lim}M_{n}$.

\begin{proposition}
\label{P_StrongAscending} Let $\mathcal{M}=(M_{n})_{n\in\mathbb{N}^{\ast}}$ be
an ascending sequence of Banach $C^{\infty}$-manifolds modeled on the
ascending sequence $(E_{n})_{n\in\mathbb{N}^{\ast}}$.
%Denote by $\bar{E}_{n}$
%the closure of $E_{n}$ in $E_{n+1}$ (provided with the induced norm of
%$E_{n+1}$).
We assume that:

(i) ${E}_{n}$ is a complemented subspace in $E_{n+1}$ for each $n\in
\mathbb{N}^{\ast}$;

(ii) there exists a linear connection\footnote{For the definition of a
connection on a Banach manifold see for instance \cite{DoGa}.
%Such a
%connection does not always exist. As in finite dimension, there exists a
%linear connection on a paracompact manifold manifold with smooth partitions of
%unity
For more details see also subsection \ref{banachconnection}.} on $TM_{n}$ for
each $n\in\mathbb{N}^{\ast}$.

%(iii)  for each $n\in \N$, either $E_{n}$ is closed in $E_{n+1}$ or  for each $n\in\N$ the inclusion of $E_{n}$ in $E_{n+1}$ is compact.

\noindent Then $E=\underrightarrow{\lim}E_{n}$ is an LB-space and
$\mathcal{M}$ has the direct limit chart property at each point of
$M=\underrightarrow{\lim}M_{n}$.
\end{proposition}

\begin{example}
\label{CSabc} By application of Proposition \ref{P_StrongAscending},
$\mathcal{M}=(M_{n})_{n\in\mathbb{N}^{\ast}}$ has the direct limit chart
property at each point in the following cases:

\begin{enumerate}
\item Each space $M_{n}$ is a paracompact finite dimensional manifold (cf.
\cite{Glo2}).

\item Each manifold $M_{n}$ is a smooth paracompact \footnote{i.e. $M_n$ is paracompact  and every locally finite open covering of  $M_n$  admits a conveniently smooth  partition  of unity subordinated to it.}  Hilbert submanifold of $M_{n+1}$.

\item Each manifold $M_{n}$ is a smooth paracompact Banach submanifold of $M_{n+1}$.

%and has a linear connection.
%\item the inclusion of $E_n$ in $E_{n+1}$ is compact with a dense range,  and each $M_n$  has a linear connection.

\end{enumerate}
\end{example}

The proof of Proposition \ref{P_StrongAscending} requires the following Lemma:

\begin{lemma}
\label{L_extchartsubcomp} Let $N_{1}$ be a  Banach complemented immersed submanifold,
modeled on $F_{1}$, of a Banach manifold $N$, modeled on $F$ with
$F_{1}\subset F$. Assume that
%$F_{1}$ is complemented in $F$ and
there exists a linear connection on $TN$. Given any chart $(U_{1},\phi_{1})$
of $x$ in $N_{1}$ such that $U_{1}$ is a contractible set, there exists a
chart $(U,\Phi)$ of $x\in N_{1}$ such that $U$ is contractible, $U\cap
N_{1}=U_{1}$ and $\Phi_{|U_{1}}=\phi_{1}$.
\end{lemma}

\begin{proof}
Let $N_{1}$ be an immersed complemented submanifold of a Banach manifold $N$.
If $N_{1}$ (resp. $N$) is modeled on $F_{1}$ (resp. $F$), there exists a Banach
subspace $F_{2}$ of $F$ such that $F=F_{1}\oplus F_{2}$.\newline Assume that
there exists a linear connection on $TN$. Therefore we have an exponential map
$\operatorname{Exp}:\mathcal{O}\subset TN\rightarrow N$ where $\mathcal{O}$ is
an open neighborhood of the zero section in $TN$. Note that
$\operatorname{Exp}_{|T_{x}N}$ is a local diffeomorphism.
\newline
Choose a chart $(U_{1},\phi_{1})$ of $N_{1}$ around $x$ such that $U_{1}$ is a
contractible set. Now, as $N_{1}$ is immersed in $N$, for each $z\in U_{1}$,
there exists a chart $(U_{z},\Phi_{z})$ of $N$ such that
\[
\Phi_{z}\left(  U_{z}\cap N_{1}\right)  =\phi_{1}\left(  U_{z}\cap
U_{1}\right)  \times\left\{  0_{F_{2}}\right\}  \text{ and }{\Phi_{z|}}%
_{U_{z}\cap U_{1}}={\phi_{1}}_{|U_{z}\cap U_{1}}%
\]
Then $U={\displaystyle\bigcup\limits_{z\in U_{1}}}U_{z}$ is an open neighborhood containing $U_{1}$ and we have
\[
U\cap N_{1}=\bigcup_{z\in U_{1}}(U_{z}\cap N_{1})=U_{1}.
\]
Therefore $U_{1}$ is a closed submanifold of $U$.\newline

As $U_{1}$ is a contractible set, the restriction of $TN$ to $U_{1}$ is
trivial (cf. \cite{AMR} Theorem 3.4.35). Therefore we have a diffeomorphism
$\Theta:TN_{|U_{1}}\rightarrow U_{1}\times F$. In the trivial bundle
$U_{1}\times F$, we can consider the subbundle $U_{1}\times F_{2}$ and we have
$TN_{|U_{1}}=TU_{1}\oplus\Theta^{-1}(U_{1}\times F_{2})$. As $\phi_{1}(U)$ is
an open set of $F_{1}$ and $F_{1}$ is paracompact, so is $U_{1}$. Therefore by
same arguments used in the proof of Theorem 5.1  chapter  IV of \cite{Lan}, we can
build a diffeomorphism $\Psi$ from an open neighborhood $\mathcal{U}$ of the
zero section of $\Theta^{-1}(U_{1}\times F_{2})$ on an open neighborhood $U$
of $U_{1}$ in $N$. Note that $\mathcal{U}$ is a fibration on the zero section
of $\Phi^{-1}(U_{1}\times F_{2})$. Moreover, from the property of
$\operatorname{Exp}$, we can choose $\mathcal{U}$ such that each fiber is a
contractible set. We denote by $\Phi$ the composition defined by $\Phi^{-1}=$
$\Theta\circ\Psi^{-1}\circ(\left(  \phi_{1}\right)  ^{-1}\times Id_{F_{2}})$.
As $\operatorname{Exp}_{|T_{u}N}(0)\left(  u\right)  =u$ we finally have
$\Phi^{-1}(v,0)=\left(  \phi_{1}\right)  ^{-1}(v)$.
\end{proof}

\begin{proof}
\textit{of Proposition \ref{P_StrongAscending}}${}$\newline
We have already seen that $E=\underrightarrow{\lim}E_{n}$ is an LB-space (cf.
Proposition \ref{P_PreservationCompletness}).

Now consider any point $x\in M=\underrightarrow{\lim}M_{n}$. Then $x$ belongs
to some $E_{n}$. Let $l_{0}$ be the first integer $l$ such that $x$ belongs to
$M_{l}$. Assume that for each integer $l_{0}\leq l\leq k$ we have the
following property: there exists a family of charts $(U_{n},\phi_{n})$ of
$M_{n}$, for each $l_{0}\leq n\leq l$, such that:
\begin{itemize}
\item $(U_{n})_{l_{0}\leq n\leq l}$ is an ascending sequence of chart domains
around $x$;

\item ${\phi_{n+1}}_{|U_{n}}=\phi_{n}$ for all $l_{0}\leq n<l$.
\end{itemize}

From Lemma \ref{L_extchartsubcomp}, this assumption is true for $l=l_{0}+1$.
The proof is obtained by induction using Lemma \ref{L_extchartsubcomp}.
\end{proof}

Now we can prove the following result which generalizes \cite{Glo2}, Theorem 3.1.

\begin{theorem}
\label{T_LBCmanifold} Let $(M_{n})_{n\in\mathbb{N}^{\ast}}$ be an ascending
sequence of Banach $C^{\infty}$-manifolds, modeled on the Banach spaces $E_{n}
$. Assume that $(M_{n})_{n\in\mathbb{N}^{\ast}}$ has the direct limit chart
property at each point $x\in M=\underrightarrow{\lim}M_{n}$ and
$E=\underrightarrow{\lim}E_{n}$ is an LB-space.

Then there is a unique  n.n.H. convenient manifold structure on
$M=\underrightarrow{\lim}M_{n}$ modeled on the convenient space $E$ such that the topology associated to this structure is the  $DL$-topology on $M$. \\
In particular, for each $n\in\mathbb{N}^{\ast}$, the canonical injection $\varepsilon_{n}:M_{n}\longrightarrow M$ is an
injective conveniently smooth map  and $(M_{n},\varepsilon_{n})$ is a weak submanifold of $M$.\\
Moreover, if each $M_n$ is locally compact or is open in $M_{n+1}$ or is a paracompact Banach manifold closed  in $M_{n+1}$, then $M=\underrightarrow{\lim}M_{n}$ is provided with a Hausdorff  convenient manifold structure.\\
\end{theorem}

A direct application of this theorem gives rise to the following result:
\begin{corollary}
\label{ascendingSequenceConnection} Let $(M_{n})_{n\in\mathbb{N}^{\ast}}$ be
an ascending sequence of Banach  paracompact $C^{\infty}$-manifolds  where $M_n$ is closed in $M_{n+1}$. If the sequence $(M_{n})_{n\in\mathbb{N}^{\ast}}$ satisfies the
assumptions of Proposition \ref{P_StrongAscending}, then $M=\underrightarrow
{\lim}M_{n}$, provided with the $DL$- topology,  has  a unique structure of Hausdorff convenient manifold modeled on an LB-space.
\end{corollary}

\begin{proof}
\textit{ of Theorem \ref{T_LBCmanifold}}.- As in Lemma \ref{L_asEn} part (ii), we consider the set $\mathcal{A}$ of all sequences of charts $\{(U_{n}%
^{\alpha},\phi_{n}^{\alpha})_{n\in\mathbb{N}^{\ast}}\}_{\alpha\in{A}}$ of
$(M_{n})_{n\in\mathbb{N}^{\ast}}$ such that $(U_{n})$ is an ascending sequence
of chart domains. We set $V_{n}^{\alpha}=\phi_{n}^{\alpha}(U_{n}^{\alpha})$,
$U^{\alpha}=\underrightarrow{\lim}U_{n}^{\alpha}$ and $\phi^{\alpha
}=\underrightarrow{\lim}\phi_{n}^{\alpha}$. From Lemma \ref{L_asEn} part (ii),
$\phi^{\alpha}$ is a homeomorphism from $U^{\alpha}$ to the open set
$V^{\alpha}=\phi^{\alpha}(U^{\alpha})=\underrightarrow{\lim}V_{n}^{\alpha}$
of $E$. Then $U^{\alpha}$ and $V^{\alpha}$ are open sets of the
$DL$-topology on $M$ and $E$ respectively (cf. Proposition
\ref{P_TopologicalPropertiesAscendingSequenceTopologicalSpaces} part 2). From
our assumption, $\mathcal{A}$ is then a topological atlas of $M$ modeled on
the convenient space $E$. Note that, from Proposition \ref{P_LBCctoplctop},
each $V^{\alpha}$ is also a $c^{\infty}$-open set.
\newline
Let us prove that the change of charts are conveniently smooth diffeomorphisms.
Consider two charts $\left(  U^{\alpha},\phi^{\alpha}\right)  $ and $\left(
U^{\beta},\phi^{\beta}\right)  $ around $x\in M$. We consider
\[
\tau_{n}^{\beta\alpha}=\phi_{n}^{\beta}\circ\left(  \phi_{n}^{\alpha}\right)
^{-1}:\phi_{n}^{\alpha}\left(  U_{n}^{\alpha}\cap U_{n}^{\beta}\right)
\longrightarrow\phi_{n}^{\beta}\left(  U_{n}^{\alpha}\cap U_{n}^{\beta
}\right).
\]
For each $n\in\mathbb{N}^{\ast}$, the pairs $(U_{n}^{\alpha},\phi_{n}^{\alpha
})$ and $(U_{n}^{\beta},\phi_{n}^{\beta})$ are charts of $M_{n}$ and the
intersection of their domains is not empty. It follows that the map $\tau
_{n}^{\beta\alpha}$ is a $C^{\infty}$ local diffeomorphism of $E_{n}$. But the
construction of $U^{\alpha}$ and $U^{\beta}\ $implies that $U^{\alpha}\cap
U^{\beta}$ is the direct limit of $(U_{n}^{a}\cap U_{n}^{\beta})_{n\in
\mathbb{N}^{\ast}}$ in $M$. It follows that $\phi^{\alpha}\left(  U^{\alpha
}\cap U^{\beta}\right)  $ is the direct limit of $\phi_{n}^{\alpha}\left(
U_{n}^{\alpha}\cap U_{n}^{\beta}\right)  _{n\in\mathbb{N}^{\ast}}$. In the
same way, we have $\phi^{\alpha}\left(  U^{\alpha}\cap U^{\beta}\right)
=\underrightarrow{\lim}\phi_{n}^{\alpha}\left(  U_{n}^{\alpha}\cap
U_{n}^{\beta}\right)  $. Therefore we get a direct limit map $\tau
^{\beta\alpha}=\underrightarrow{\lim}\tau_{n}^{\beta\alpha}$ from the open set
$\phi^{\alpha}\left(  U^{\alpha}\cap U^{\beta}\right)  $ onto the open set
$\phi^{\beta}\left(  U^{\alpha}\cap U^{\beta}\right)  $ of $E$. Again the sets
$\phi^{\alpha}\left(  U^{\alpha}\cap U^{\beta}\right)  $ and $\phi^{\beta
}\left(  U^{\alpha}\cap U^{\beta}\right)  $ are $c^{\infty}$-open sets of $E$.
As each $\tau_{n}^{\beta\alpha}$ is a $C^{\infty}$ diffeomorphism of $E_{n}$,
Lemma \ref{L_Cinftycinfty} implies that $\tau^{\alpha\beta}$ is a conveniently smooth
diffeomorphism from $\phi^{\alpha}\left(  U^{\alpha}\cap U^{\beta}\right)  $
onto $\phi^{\beta}\left(  U^{\alpha}\cap U^{\beta}\right)  $. Therefore we
obtain that $\mathcal{A}$ is  convenient atlas on $M.$ Note that the topology of $M$ defined by such an atlas is exactly the $DL$-topology  on $M$. Therefore, if $M_n$ locally compact, then $M$ is Hausdorff from Proposition \ref{P_TopologicalPropertiesAscendingSequenceTopologicalSpaces}. In the same way, from
Proposition \ref{directlimitT4}, if $M_n$ is open in $M_{n+1}$ or is paracompact and closed in $M_{n+1}$ for each $n$, then $M$ is Hausdorff. Thus in each of the previous particular cases, $M$ is provided with a  Hausdorff convenient manifold structure.\\

Now we prove the uniqueness of this convenient structure. Assume that $Y$ is
a convenient manifold structure  modeled on the convenient vector space $E$ and
$h_{n}:M_{n}\longrightarrow Y$ a $C^{\infty}$ map for each $n\in
\mathbb{N}^{\ast}$ s.t. $\left(  Y,\left(  h_{n}\right)  _{n\in\mathbb{N}%
^{\ast}}\right)  $ is a cone over $\mathcal{S}$. Then there is a uniquely
determined continuous map $h:M\longrightarrow Y$ s.t. $h_{|M_{n}}=h_{n}$. Let
$x\in M$; we can find a chart $\left(  U,\varphi\right)  $ around $x$ in the
atlas $\mathcal{A}$ where $f=\underrightarrow{\lim}f_{n}U$ for charts
$\varphi_{n}:$ $U_{n}\longrightarrow E_{n}$. Let $\psi:W\subset
Y\longrightarrow V$ be a chart for $Y$ ($W$ is an open set of $Y$). Then
$O=\left(  h\circ\varphi^{-1}\right)  ^{-1}\left(  W\right)  $ is an open set
of $U\subset E$ and $O_{n}=O\cap E_{n}$ is open in $E_{n}$ for each $n.$
Consider $g=\psi\circ h\circ\varphi^{-1}|_{O}^{V}:O\longrightarrow V$. Then
$g_{|O_{n}}=\psi\circ h_{n}\circ\varphi_{n}^{-1}|_{O_{n}}^{V}:O_{n}%
\longrightarrow V$ is $C^{\infty}$ for each $n\in\mathbb{N}^{\ast}$. Hence $g$
is $c^{\infty}$ (cf. Proposition \ref{P_DirectLimitMap}), so is $h$ on
the open neighborhood $U$ of $x$ and hence on all of $M$ because $x$ is
arbitrary. Thus $\left(  M,\left(  \varepsilon_{n}\right)  _{n\in N^{\ast}%
}\right)  =\underrightarrow{\lim}S$ in the category of $c^{\infty}$-manifolds.
The uniqueness of a convenient structure of manifold on $M$ follows from the
universal property of direct limits.
%Finally, any smooth curve $\gamma:\mathbb{R}\rightarrow M_{n}$ is contained in
%$M$. Therefore, each canonical injection $\varepsilon_{n}:M_{n}\longrightarrow
%M$ is injective and $c^{\infty}$. Moreover according to the definition of the
%kinematic tangent space $T_{x}M$ to a convenient manifold $M$ at a point $x$,
%it is clear that for $x\in M_{n}$, the tangent space $T_{x}M_{n}$ is a
%subspace of $T_{x}M$.
%Now assume that each $M_n$ is a paracompact Banach manifold and is closed in $M_{n+1}$ for each $n$. Since the topology on $M$ defined by the atlas of convenient is   direct limit on $M$, according to Propositon \ref{directlimitT4} it follows that $M$ is a  convenient manifold.
\end{proof}

\subsection{\label{*DL_LieGroups}Direct limit of Lie groups}

The reader is referred to \cite{Glo3}. \newline
Interesting infinite-dimensional Lie groups often appears as direct limits $G=\bigcup
\limits_{n\in\mathbb{N}^{\ast}}G_{n}$ of ascending Lie groups $G_{1}\subset
G_{2}\subset\cdots$ where the bonding maps (inclusion maps $\varepsilon
_{n}^{n+1}:G_{n} \longrightarrow G_{n+1}$) are smooth homomorphisms (e.g. the
group $\operatorname{Diff}_{c}\left(  M\right)  $ of compactly supported
diffeomorphisms of a $\sigma-$compact smooth manifold $M$ or the \textit{test
function} groups $C_{c}^{\infty}\left(  M,H\right)  $ of compactly supported
smooth maps with values in a finite-dimensional Lie group $H$).

When the Lie groups $G_{n}$ are finite dimensional it is well known that the
direct limit $\underrightarrow{\lim}G_{n}$ can be endowed with a structure of
Lie group (see \cite{Glo2}).

Here we give conditions on the direct sequences $\mathcal{G}=\left(
G_{n},\varepsilon_{n}^{m}\right)  _{n\in\mathbb{N}^{\ast},\ m\in
\mathbb{N}^{\ast},\ n\leq m}$ of Lie groups in order to obtain a structure of
Lie group on their direct limit.

\bigskip

We first recall the essential notion of \textit{candidate for a direct limit
chart}:

\begin{definition}
\label{D_CandidateDirectLimitChart_LieGroup}Let $G=\bigcup\limits_{n\in
\mathbb{N}^{\ast}}G_{n}$ be the union of an ascending sequence of $C^{\infty
}$-Lie groups $G_{n}$ where the inclusion maps $\varepsilon_{n}^{m}%
:G_{n} \longrightarrow G_{m}$ are $C^{\infty}$-homomorphisms and $G_{n}$ is a
subgroup of $G$.
\newline
 We say that $G$ has a candidate for a direct
limit chart if there exist charts $\phi_{n}:G_{n}\supset U_{n} \longrightarrow
V_{n}\subset\mathfrak{g}_{n}$ of $G_{n}$ around the identity for
$n\in\mathbb{N}^{\ast}$(where $\mathfrak{g}_{n}$ stands for the Lie algebra of
$G_{n}$) such that $U_{n}\subset U_{m}$ and $\phi_{m}|_{Un}=\mathbb{L}\left(
i_{n}^{m}\right)  \circ\phi_{n}$ if $n\leq m$ and $V=\bigcup\limits_{n\in
\mathbb{N}^{\ast}}V_{n}$ is open in the locally convex direct limit
$\underrightarrow{\lim}\mathfrak{g}_{n}$ which we assume to be Hausdorff.
\end{definition}

Gl\"{o}ckner obtains the following result ({cf.} \cite{Glo3} Proposition 1.4.3):

\begin{proposition}
\label{P_DLAscendingSequence-BanachLieGroups}Let $G=\bigcup\limits_{n\in
\mathbb{N}^{\ast}}G_{n}$ be a group which is the union of an ascending
sequence of $C^{\infty}$-Lie groups. Assume that $G$ has a candidate
$\phi:U \longrightarrow V$ $\subset\underrightarrow{\lim}\mathfrak{g}_{n}$ for
a direct limit chart and assume that one of the following conditions is satisfied:

(i) $G_{n}$ is a Banach Lie group for each $n\in\mathbb{N}^{\ast}$ and the
inclusion map $\mathfrak{g}_{n} \longrightarrow \mathfrak{g}_{m}$ is a compact
linear operator for all $n<m$;

(ii) $\mathfrak{g}_{n}$ is a $k_{\omega}-$space for each $n\in\mathbb{N}%
^{\ast}$.
\newline
Then on $G=$ $\underrightarrow{\lim}G_{n}$ there exists  a
unique $C^{\infty}$-Lie group structure making  $\phi_{| W}$ a direct chart limit for $G$ around $1$, where
$W$ is an open neighbourhood of $1$ contained in $U$.
\end{proposition}

Gl\"{o}ckner gives also results in the convenient setting where $V=\bigcup
\limits_{n\in\mathbb{N}^{\ast}}V_{n}$ is a $c^{\infty}$-open set in
$\underrightarrow{\lim}\mathfrak{g}_{n}$ endowed with a suitable locally convex topology ({cf.} \cite{Glo3}, version arXiv:math/0606078, Remark 14.8):

\begin{proposition}
\label{P_DLAscendingSequencecinfty-LieGroups}Let $G=\bigcup\limits_{n\in
\mathbb{N}^{\ast}}G_{n}$ be a group which is the union of an ascending
sequence of (Hausdorff) convenient Lie groups. Equip the vector space $\underrightarrow
{\lim}\mathfrak{g}_{n}$ with the locally convex vector topology associated
with the direct limit bornology which is assumed to be Hausdorff. We require
that $G$ admits a candidate for a direct limit chart in the convenient sense
and that each bounded subset in $\mathfrak{g}$ is a bounded subset of some
$\mathfrak{g}_{n}$.\newline Then $G$ may be endowed with a structure  n.n.H. convenient Lie group %of  convenient
%Lie group\footnote{A priori we get only an n.n.H. convenient Lie group. But since each $G_n$ is $T1$ it follows that $G$ is also $T1$ and then each point is closed in $G$. Now   it follows from \cite{Bou1} that $G$ is Hausdorff.}
\end{proposition}
This criterion permits to obtain:

\begin{theorem}
\label{T_LBCLiegroup} Let $G=\bigcup\limits_{n\in\mathbb{N}^{\ast}}G_{n}$ be a
group which is the union of an ascending sequence of Banach Lie groups. Assume
that the direct limit $\mathfrak{g}=\underrightarrow{\lim}\mathfrak{g}_{n}$ of the
ascending sequence $(\mathfrak{g}_{n})_{n\in\mathbb{N}^{\ast}}$ of associated
Lie algebras is an LB-space. If $G$ admits a candidate for a direct limit
chart, then $G$ can be endowed with a structure of n.n.H. convenient Lie group modeled on the
LB-space $\mathfrak{g}$.
\end{theorem}

Therefore, Proposition \ref{P_DLAscendingSequence-BanachLieGroups} assumption
(ii) can be seen as a corollary of this theorem. According to Proposition
\ref{P_Mixtes_a-b}, if $G$ admits a candidate for a direct limit chart then
assume that there exists a countable subset $I\subset\mathbb{N}^{\ast}$
such that the direct limit $\mathfrak{g}_{I}=\underrightarrow{\lim
}\mathfrak{\{}\mathfrak{g}_{i},i\in I\}$ is an LB-space, then $G$ is endowed
with a structure of  convenient Lie group modeled on the LB-space
$\mathfrak{g}$. Now  since $\mathfrak{g}_{n}$ is a $k_{\omega}-$space for each $n\in\mathbb{N}^*$,  by  direct chart limit property and  using the fact that direct limits of ascending sequences of locally $k_\omega$-spaces are locally $k_\omega$-spaces  and so are Hausdorff (see \cite{Glo3}), the topology on $G$ must be Hausdorff.\newline

On the other hand, the reader can find the following criterion in \cite{Dah}:

\begin{theorem}
\label{T_Liegroup} Let $\mathcal{G}=(G_{n})_{n\in\mathbb{N}^{\ast}}$ be an
ascending sequence of Banach Lie groups such that all inclusion maps
$j_{n}:G_{n}\rightarrow G_{n+1}$ are analytic group morphisms and assume that
we have the following properties:

\begin{enumerate}
\item[(a)] For each $n\in\mathbb{N}^{\ast}$, there exists a norm $||\;||_{n}$
on the Lie algebra $\mathfrak{g}_{n}$ defining its Banach space structure,
such that its Lie bracket satisfies the inequality $||[x,y]||_{n}\leq
||x||_{n}||y||_{n}$ for all $x$ and $y$ in $\mathfrak{g}_{n}\mathbb{\ }$and
such that the bounded linear operator $\mathbb{L}(j_{n}):\mathfrak{g}%
_{n}\rightarrow\mathfrak{g}_{n+1}$ has a norm operator bounded by $1$;

\item[(b)] The locally convex structure of vector space $\mathfrak{g}%
=\underrightarrow{\lim}\mathfrak{g}_{n}$ is Hausdorff;

\item[(c)] The map $\exp_{G}=\bigcup\limits_{n\in\mathbb{N}^{\ast}}\exp
_{G_{n}}:\bigcup\limits_{n\in\mathbb{N}^{\ast}}\mathfrak{g}_{n}\longrightarrow
\bigcup\limits_{n\in\mathbb{N}^{\ast}}G_{n}$ is injective on some neighborhood
of $0$.
\end{enumerate}
Then $G=\underrightarrow{\lim}G_{n}=\bigcup\limits_{n\in
\mathbb{N}^{\ast}}G_{n}$ has an analytic structure of Lie group modeled on
$\mathfrak{g}$ and $\exp_{G}$ is an analytic diffeomorphism from some
neighborhood of $0$ to a neighborhood of $1\in G$.

\end{theorem}

We end this subsection with an application of this result\footnote{This result
is certainly well known by specialists but it is an easy corollary of Theorem
\ref{T_Liegroup} and so we give a proof here.}:

\begin{theorem}
\label{T_ConvenientGL} Let $(E_{n})_{n\in\mathbb{N}^{\ast}}$ be an ascending
sequence of Banach spaces such that $E_{n}$ is a complemented Banach subspace
of $E_{n+1}$. \noindent Then $E=\bigcup\limits_{n\in\mathbb{N}^{\ast}}E_{n}$
is an LB-space and $L(E)=\bigcup\limits_{n\in\mathbb{N}^{\ast}}L(E_{n})$ is
also an LB-space, where $L(E_{n})$ is the Banach space of continuous linear
operators of $E_{n}$. Moreover, $GL(E)=\bigcup\limits_{n\in\mathbb{N}^{\ast}%
}GL(E_{n})$ has a structure of convenient Lie group modeled on $L(E)$, where
$GL(E_{n})$ is the Banach Lie group of linear continuous automorphisms of
$E_{n}$.
\end{theorem}

For the proof of this theorem we need the following lemma:

\begin{lemma}
\label{L_inclusionLEn} Let $E$ and $F$ be two Banach spaces such that $E$ is a
complemented Banach subspace of $F$.
%Given any norm $||\;||_E$ on $E$ there exists a norm $||\;||_F$ on  $F$ such that $||\iota(x)||_F\leq||x||_E$ and
Given a norm $||\;||_{E}$ on $E$, there exists a norm $||\;||_{F}$ on $F$ and
an embedding $\lambda:L(E)\longrightarrow L(F)$ which is an isometry with
respect to the corresponding operator norms on $L(E)$ and $L(F)$ respectively.
\newline Moreover, we have $[\lambda(T),\lambda(T^{\prime})]=\lambda
([T,T^{\prime}])$ where, as classically, the bracket is given by
$[T,T^{\prime}]$ $=T\circ T^{\prime}-T^{\prime}\circ T$.
\end{lemma}

\begin{proof}
Let $E^{\prime}$ be a subspace of $F$ such that $F=E\oplus E^{\prime}$. We
endow $E^{\prime}$ with a norm $||\;||^{\prime}$ and let $||\;||_{F}$ be the
norm on $F$ defined by $||x||_{F}=||x_{1}||_{E}+||x_{2}||^{\prime}$ if
$x=x_{1}+x_{2}$ with $x_{1}\in E$ and $x_{2}\in E^{\prime}$. Denote by
$\lambda$ the natural inclusion of $E$ in $F$. By construction, $\lambda$ is an
isometry. We define $\Lambda:L(E)\longrightarrow L(F)$ where $\Lambda(T)$ is
the operator on $F$ whose restriction to $E$ is $T$ and whose restriction to
$E^{\prime}$ is the null operator. Clearly $\Lambda$ is injective and the
operator norm of $\Lambda$ is $1$. Indeed if $\Pi$ is the projection of $F$ on
$E$ with kernel $E^{\prime}$, we have
\[
\frac{||\Lambda(T)(x)||_{F}}{||x||_{F}}\leq\frac{||T\circ\Pi(x)||_{F}}%
{||\Pi(x)||_{E}}\leq||T||_{L(E)}.
\]
We deduce $||\Lambda(T)||_{L(F)}\leq||T||_{L(E)}$. On the other hand%
\begin{align*}
||T||_{L(E)}  & =\sup\{\frac{||T(x)||_{E}}{||x||_{E}},x\in E\}\\
& \leq\sup\{\frac{||T(x)||_{E}}{||x||_{F}},x\in F\}=\sup\{\frac{||\Lambda
(T)(x)||_{F}}{||x||_{F}},x\in F\}.
\end{align*}
Finally, it is easy to verify that we have $\Lambda(T\circ T^{\prime}%
)=\Lambda(T)\circ\Lambda(T^{\prime})$, which ends the proof.
\end{proof}

\begin{proof}
of Theorem \ref{T_ConvenientGL}.--
According to Lemma \ref{L_inclusionLEn}, by
induction, we can build a sequence of norms $||\;||_{n}$ on each $E_{n}$ and an
isometry $\Lambda_{n}:L(E_{n})\longrightarrow L\left(  E_{n+1}\right)  $. For
simplicity, we identify $L(E_{n})$ with $\Lambda_{n}(L(E_{n}))$ in $L\left(
E_{n+1}\right)  $. Then $L(E_{n})$ is a Banach subspace of $L(E_{n+1})$ with
the induced topology. It follows that $\mathfrak{G}=\bigcup_{n\in\mathbb{N}%
}L(E_{n})=\underrightarrow{\lim}L(E_{n})$ is a convenient space. On the other
hand, for the operator norm in each $L(E_{n})$ we have
\[
||[T,T^{\prime}]||_{L(E_{n})}\leq2||T||_{L(E_{n})}||T^{\prime}||_{L(E_{n})}%
\]
On each $L(E_{n})$, we consider the norm $\nu_{n}={2}||\;||_{L(E_{n})}$. Then
$\nu_{n}$ defines the topology of $L(E_{n})$. The inclusion $\Lambda_{n}$ is
still an isometry and we have
\[
\nu_{n}([T,T^{\prime}])\leq\nu_{n}(T)\nu_{n}(T^{\prime}).
\]

Given $T\in\mathfrak{G}$, then $T$ belongs to some $L(E_{n})$; we then have
\[
\exp_{\mathfrak{G}}(T)=\sum_{k\in\mathbb{N}}\dfrac{T^{k}}{k!}%
\]
On one hand, classically, the exponential map $\exp_{n}:L(E_{n}%
)\longrightarrow GL(E_{n})$ is an analytic diffeomorphism over the ball
$B_{n}(0,\ln(2))$ (relative to the norm $\nu_{n}$ on $L(E_{n})$) in
$GL(E_{n})$. On the other hand, we have the relations: ${\exp_{n+1}}%
_{|L(E_{n})}=\exp_{n}$ and $B_{n+1}(0,\frac{1}{2}\ln(2))\bigcap E_{n}%
=B_{n}(0,\frac{1}{2}\ln(2))$.\\
It follows that $\exp_{\mathfrak{G}}$ is injective on $\bigcup_{n\in\mathbb{N}}B_{n}(0,\frac{1}{2}\ln(2))$.\\
Therefore all the assumptions of Theorem \ref{T_Liegroup} are satisfied and we get the
announced result.
\end{proof}

\subsection{\label{*DL_BanachVectorBundles}Direct limit of Banach vector bundles}

\begin{definition}
\label{D_AscendingSequenceBanachVectorBundles} A sequence $\left( E_{n},\pi
_{n},M_{n}\right) _{n\in \mathbb{N}^{\ast }}$ of Banach vector bundles is
called a strong ascending sequence of Banach vector bundles if the following
assumptions are satisfied:\newline
1. $\mathcal{M}=(M_{n})_{n\in \mathbb{N}^{\ast }}$ is an ascending sequence
of Banach $C^{\infty }$-manifolds, where $M_{n}$ is modeled on the Banach
space $\mathbb{M}_{n}$ such that $\mathbb{M}_{n}$ is a complemented Banach
subspace of $\mathbb{M}_{n+1}$ and the inclusion $\varepsilon
_{n}^{n+1}:M_{n}\longrightarrow M_{n+1}$ is a $C^{\infty }$ injective map
such that $(M_{n},\varepsilon _{n}^{n+1})$ is a weak submanifold of $M_{n+1}$%
; \newline
2. The sequence $(E_{n})_{n\in \mathbb{N}^{\ast }}$ is an ascending sequence
such that the sequence of typical fibers $\left( \mathbb{E}_{n}\right)
_{n\in \mathbb{N}^{\ast }}$ of $(E_{n})_{n\in \mathbb{N}^{\ast }}$ is an
ascending sequence of Banach spaces such that $\mathbb{E}_{n}$ is a
complemented Banach subspace of $\mathbb{E}_{n+1}$;\newline
3. For each $n\in \mathbb{N}^{\ast }$, $\pi _{n+1}\circ \lambda _{n}^{n+1}=\varepsilon _{n}^{n+1}\circ \pi _{n}$ where $\lambda _{n}^{n+1}:E_{n}\longrightarrow E_{n+1}$ is the natural
inclusion;\newline
4. Any $x\in M=\underrightarrow{\lim }M_{n}$ has the direct limit chart
property for $(U=\underrightarrow{\lim }U_{n},\phi =\underrightarrow{\lim }%
\phi _{n})$;\newline
5. For each $n\in \mathbb{N}^{\ast }$, there exists a trivialization $\Psi
_{n}:\left( \pi _{n}\right) ^{-1}\left( U_{n}\right) \longrightarrow
U_{n}\times \mathbb{E}_{n}$ such that the following diagram is commutative:

\begin{equation*}
\begin{array}{ccc}
\left( \pi _{n}\right) ^{-1}\left( U_{n}\right) & \underrightarrow{\lambda
_{n}^{n+1}} & \left( \pi _{n+1}\right) ^{-1}\left( U_{n+1}\right) \\ 
\Psi _{n}\downarrow &  & \downarrow \Psi _{n+1} \\ 
U_{n}\times \mathbb{E}_{n} & \underrightarrow{\left( \varepsilon
_{n}^{n+1}\times \iota _{n}^{n+1}\right) } & U_{n+1}\times \mathbb{E}_{n+1}.%
\end{array}%
\end{equation*}

\end{definition}

For example, the sequence $\left( TM_{n},\pi _{n},M_{n}\right) _{n\in 
\mathbb{N}^{\ast }}$ is a strong ascending sequence of Banach vector bundles
whenever $(M_{n})_{n\in \mathbb{N}^{\ast }}$ is an ascending sequence which
has the direct limit chart property at each point of $x\in M=%
\underrightarrow{\lim }M_{n}$ whose model $\mathbb{M}_{n}$ is complemented
in $\mathbb{M}_{n+1}$.

\begin{proposition}
\label{P_StructureOnDirectLimitLinearBundles} Let $\left( E_{n},\pi
_{n},M_{n}\right) _{n\in \mathbb{N}^{\ast }}$ be a strong ascending sequence
of Banach vector bundles. We have:

1. $\underrightarrow{\lim}E_{n}$ has a structure of n.n.H convenient
manifold modeled on the LB-space $\underrightarrow{\lim}\mathbb{M}%
_{n}\times \underrightarrow{\lim}\mathbb{E}_n$ which has a Hausdorff
convenient structure if and only if $M$ is Hausdorff.

2. $\left( \underrightarrow{\lim }E_{n},\underrightarrow{\lim }\pi _{n},%
\underrightarrow{\lim }M_{n}\right) $ can be endowed with a structure of convenient vector bundle whose typical fiber is $\underrightarrow{\lim }\mathbb{\mathbb{%
E}}_{n}$ and whose structural group is a Fr\'{e}chet topological group.
\end{proposition}

Proof.-- 1. Consider $(x,v)$ in some $E_{n}$; in particular $x$ belongs to $%
M_{n}$. According to the assumptions 3. and 4, there exists a chart $\left(
U,\phi \right) $ of $M=\varinjlim M_{i}$ around $x$ where $\phi
:U=\bigcup\limits_{i\geq n}U_{i}\longrightarrow V=\bigcup\limits_{i\geq
n}\phi _{i}\left( U_{i}\right) $, $\left( U_{i},\phi _{i}\right) $ being a
chart around $x_{i}$ and $V_{i}=\phi _{i}\left( U_{i}\right) \subset \mathbb{%
M}_{i}$. A local trivialization $\Psi _{i}:\pi _{i}^{-1}\left( U_{i}\right)
\longrightarrow U_{i}\times \mathbb{E}_{i}$ gives rise, via the chart $\phi
_{i}:U_{i}\longrightarrow V_{i}$, to a chart $\psi _{i}:\pi _{i}^{-1}\left(
U_{i}\right) \longrightarrow V_{i}\times \mathbb{E}_{i}\subset \mathbb{M}%
_{i}\times \mathbb{E}_{i}$. \newline

From the assumption 5., we get the commutativity of the diagram 
\begin{equation*}
\begin{array}{ccc}
\left( \pi _{i}\right) ^{-1}\left( U_{i}\right) & \underrightarrow{\lambda
_{i}^{j}} & \left( \pi _{j}\right) ^{-1}\left( U_{j}\right) \\ 
\Psi _{i}\downarrow &  & \downarrow \Psi _{j} \\ 
U_{i}\times \mathbb{E}_{i} & \underrightarrow{\left( \varepsilon
_{i}^{j}\times \iota _{i}^{j}\right) } & U_{j}\times \mathbb{E}_{j}%
\end{array}%
\end{equation*}%
The previous arguments imply that the sequence of Banach manifolds $%
\{E_{n}\}_{n\in \mathbb{N}^{\ast }}$ has the direct limit chart property at
any point $(x,v)\in \underrightarrow{\lim }E_{n}$. Therefore, from Theorem %
\ref{T_LBCmanifold}, there exists a unique structure of n.n.H. convenient
manifold on $E=\underrightarrow{\lim }E_{n}$ whose topology coincides with
the $DL$-topology on $E$. In particular this structure is Hausdorff if and
only if $M$ is so. This ends the proof of part 1.\newline

2. The main difficulty is to define the structural group\footnote{As the referee pointed out, the structural group is much larger than the direct limit of the linear groups $GL\left( \mathbb{E}_{n}\right)$.}, say $G\left( \mathbb{E}\right) $ where $\mathbb{E}=\underrightarrow{\lim }\mathbb{E}_{n}$.\newline
Let $\mathbb{E}_{1}\subset \mathbb{E}_{2}\subset \cdots $ be the direct
sequence of complemented Banach spaces associated to the direct sequence $%
E_{1}\subset E_{2}\subset \cdots $; so there exist Banach subspaces $\mathbb{%
E}_{1}^{\prime },\mathbb{E}_{2}^{\prime },\dots $ such that: 
\begin{equation*}
\left\{ 
\begin{array}{c}
\mathbb{E}_{1}=\mathbb{E}_{1}^{\prime }, \\ 
\forall i\in \mathbb{N}^{\ast },\mathbb{E}_{i+1}\backsimeq \mathbb{E}%
_{i}\times \mathbb{E}_{i+1}^{\prime }%
\end{array}%
\right. \;
\end{equation*}%
For $i,j\in \mathbb{N}^{\ast },\;i\leq j$, we have the injection 
\begin{equation*}
\begin{array}{cccc}
\iota _{i}^{j}: & \mathbb{E}_{i}\backsimeq \mathbb{E}_{1}^{\prime }\times
\cdots \times \mathbb{E}_{i}^{\prime } & \rightarrow & \mathbb{E}%
_{j}\backsimeq \mathbb{E}_{1}^{\prime }\times \cdots \times \mathbb{E}%
_{j}^{\prime } \\ 
& (x_{1}^{\prime },\dots ,x_{i}^{\prime }) & \mapsto & (x_{1}^{\prime
},\dots ,x_{i}^{\prime },0,\dots ,0)%
\end{array}%
\end{equation*}%
\smallskip
Any $A_{n+1}\in GL\left( \mathbb{E}_{n+1}\right) $ is represented by $\left( 
\begin{array}{cc}
A_{n} & B_{n+1} \\ 
A_{n}^{\prime } & B_{n+1}^{\prime }%
\end{array}%
\right) $ where%
\begin{equation*}
A_{n}\in \mathcal{L}\left( \mathbb{E}_{n},\mathbb{E}_{n}\right) ,\
A_{n}^{\prime }\in \mathcal{L}\left( \mathbb{E}_{n},\mathbb{E}_{n+1}^{\prime
}\right) ,\ B_{n+1}\in \mathcal{L}\left( \mathbb{E}_{n+1}^{\prime },\mathbb{E%
}_{n}\right) \text{ and }B_{n+1}^{\prime }\in \mathcal{L}\left( \mathbb{E}%
_{n+1}^{\prime },\mathbb{E}_{n+1}^{\prime }\right) \text{.}
\end{equation*}

The group 
\begin{equation*}
GL_{0}\left( \mathbb{E}_{n+1}|\mathbb{E}_{n}\right) =\left\{ A\in GL\left( 
\mathbb{E}_{n+1}\right) :A\left( \mathbb{E}_{n}\right) =\mathbb{E}%
_{n}\right\}
\end{equation*}

can be identified with the Banach-Lie sub-group of operators of type $\left( 
\begin{array}{cc}
A_{n} & B_{n+1} \\ 
0 & B_{n+1}^{\prime }%
\end{array}%
\right) $ (cf. \cite{ChSt}).

The set 
\begin{equation*}
G_{n}=\left\{ A_{n}\in GL(\mathbb{E}_{n}):\forall k\in \left\{ 1,\dots
,n-1\right\} ,A_{n}(\mathbb{E}_{k})=\mathbb{E}_{k}\right\}
\end{equation*}%
can be endowed with a structure of Banach-Lie subgroup.\newline
An element $A_{n}$ of $G_{n}$ can be seen as 
\begin{equation*}
\centering A_{n}=\left( \;%
\begin{tabular}{clcclll|l}
\cline{3-4}\cline{6-6}\cline{8-8}
$A_{1}$ & \multicolumn{1}{c|}{$B_{2}$} & \multicolumn{1}{c|}{%
\multirow{2}{*}{$B_{3}$}} & \multicolumn{1}{c|}{\multirow{3}{*}{$B_{4}$}} & 
\multicolumn{1}{l|}{} & \multicolumn{1}{l|}{\multirow{5}{*}{$B_{i}$}} &  & 
\multicolumn{1}{l|}{\multirow{7}{*}{$B_{n}$}} \\ 
0 & \multicolumn{1}{c|}{$B^\prime_{2}$} & \multicolumn{1}{c|}{} & 
\multicolumn{1}{c|}{} & \multicolumn{1}{l|}{} & \multicolumn{1}{l|}{} &  & 
\multicolumn{1}{l|}{} \\ \cline{1-3}
\multicolumn{2}{|c|}{0} & \multicolumn{1}{c|}{$B^\prime_{3}$} & 
\multicolumn{1}{c|}{} & \multicolumn{1}{l|}{} & \multicolumn{1}{l|}{} &  & 
\multicolumn{1}{l|}{} \\ \cline{1-4}
\multicolumn{3}{|c|}{0} & $B^\prime_{4}$ & \multicolumn{1}{l|}{} & 
\multicolumn{1}{l|}{} &  & \multicolumn{1}{l|}{} \\ \cline{1-3}
\multicolumn{1}{l}{} &  & \multicolumn{1}{l}{} & \multicolumn{1}{l}{} & 
\multicolumn{1}{l|}{$\ddots$} & \multicolumn{1}{l|}{} &  & 
\multicolumn{1}{l|}{} \\ \cline{1-6}
\multicolumn{5}{|c|}{0} & $B^\prime_{i}$ &  & \multicolumn{1}{l|}{} \\ 
\cline{1-5}
\multicolumn{1}{l}{} &  & \multicolumn{1}{l}{} & \multicolumn{1}{l}{} &  & 
& $\ddots$ & \multicolumn{1}{l|}{} \\ \hline
\multicolumn{7}{|c|}{0} & $B^\prime_{n}$ \\ \cline{1-7}
\end{tabular}
\ \;\right)
\end{equation*}

\bigskip 
For $1\leq i\leq j\leq k$, we consider the following diagram

\begin{equation*}
\begin{array}{ccc}
\mathbb{E}_{k} & \underrightarrow{A_{k}} & \mathbb{E}_{k} \\ 
\iota _{j}^{k}\uparrow &  & \downarrow P_{j}^{k} \\ 
\mathbb{E}_{j} & \underrightarrow{A_{j}} & \mathbb{E}_{j} \\ 
\iota _{i}^{j}\uparrow &  & \downarrow P_{i}^{j} \\ 
\mathbb{E}_{i} & \underrightarrow{A_{i}} & \mathbb{E}_{i}%
\end{array}%
\end{equation*}

where $P_{i}^{j}:\mathbb{E}_{j}\longrightarrow\mathbb{E}_{i}$ is the
projection along the direction $\mathbb{E}_{i+1}^{\prime}\oplus\cdots \oplus%
\mathbb{E}_{j}^{\prime}$.

The map%
\begin{equation*}
\begin{array}{cccc}
\theta _{i}^{j}: & G_{j} & \longrightarrow & G_{i} \\ 
& A_{j} & \mapsto & P_{i}^{j}\circ A_{j}\circ \iota _{i}^{j}%
\end{array}%
\end{equation*}
is perfectly defined and we have:

\begin{equation*}
\begin{array}{l}
\left( \theta _{i}^{j}\circ \theta _{j}^{k}\right) \left( A_{k}\right)
=\theta _{i}^{j}\left[ \theta _{j}^{k}\left( A_{k}\right) \right] =\theta
_{i}^{j}\left( P_{j}^{k}\circ A_{j}\circ \iota _{j}^{k}\right)
=P_{i}^{j}\circ P_{j}^{k}\circ A_{j}\circ \iota _{j}^{k}\circ \iota _{i}^{j}%
\end{array}%
\end{equation*}

Because $P_{i}^{j}\circ P_{j}^{k}=P_{i}^{k}$ (projective system) and $\iota
_{j}^{k}\circ \iota _{i}^{j}=\iota _{i}^{k}$ (inductive system), we have%
\begin{equation*}
\left( \theta _{i}^{j}\circ \theta _{j}^{k}\right) \left( A_{k}\right)
=P_{i}^{k}\circ A_{j}\circ \iota _{i}^{k}=\theta _{i}^{k}\left( A_{k}\right)
\end{equation*}

So $\left( G_{i},\theta _{i}^{j}\right) _{i\leq j}$ is a projective system
of Banach-Lie groups and the projective limit $G\left( \mathbb{E}\right) =%
\underleftarrow{\lim }G_{n}$ can be endowed with a structure of Fr\'{e}%
chet topological group.

\medskip 
From assumptions 3. and 4. it follows that we have a well defined
conveniently smooth projection $\pi =\underrightarrow{\lim }\pi _{i}:%
\underrightarrow{\lim }E_{i}\longrightarrow \underrightarrow{\lim }M_{i}$
given by $\pi (x,v)=x$ and, with the previous notations, we also have $%
\underrightarrow{\lim }\left( \pi _{i}\right) ^{-1}\left( U_{i}\right) =\pi
^{-1}\left( \underrightarrow{\lim }U_{i}\right) $. \newline

The map $\Psi_{i}:\pi_{i}^{-1}\left( U_{i}\right) \longrightarrow U_{i}\times%
\mathbb{E}_{i}$ can be written $\Psi_i(y_i,u_i)=(y_i, \tilde{\Psi}%
_i(y_i)(u_i))$ where $u_i\mapsto \tilde{\Psi}_i(y_i)(u_i)$ is an isomorphism
of Banach spaces from $\pi_i^{-1}(y_i)$ to $\mathbb{E}_i$.

Consider an atlas $\mathcal{A} =\{(U^\alpha=\underrightarrow{\lim}%
U^\alpha_{i},\phi^\alpha=\underrightarrow{\lim}\phi^\alpha_{i})\}_{\alpha\in
A}$ on $M$. From the proof of the first part, the set $\hat{\mathcal{A}}%
=\{(\pi^{-1}(U^\alpha)=\underrightarrow{\lim}\pi_i^{-1}(U^\alpha_{i}),\psi^%
\alpha=\underrightarrow{\lim}\psi^\alpha_{i})\}_{\alpha\in A}$ is an atlas
for the manifold $E$.

Now, if $U_{i}^{\alpha }\cap U_{i}^{\beta }\not=\emptyset $, 
\begin{equation*}
\Psi _{i}^{\alpha }\circ (\Psi _{i}^{\beta })^{-1}:\phi _{i}^{\beta
}(U_{i}^{\alpha }\cap U_{i}^{\beta })\times \mathbb{E}_{1}^{\prime }\times
\cdots \times \mathbb{E}_{i}^{\prime }\rightarrow \phi _{i}^{\alpha
}(U_{i}^{\alpha }\cap U_{i}^{\beta })\times \mathbb{E}_{1}^{\prime }\times
\cdots \times \mathbb{E}_{i}^{\prime }
\end{equation*}%
can be written 
\begin{equation*}
(\overline{y}_{i}^{\beta },\overline{u^{\prime }}_{1}^{\beta },\dots ,%
\overline{u^{\prime }}_{i}^{\beta })\mapsto (\phi _{i}^{\alpha }\circ (\phi
_{i}^{\beta })^{-1}(\overline{y}_{i}^{\beta }),[\tilde{\Psi}_{i}^{\alpha
}(y_{i})]\circ \lbrack \tilde{\Psi}_{i}^{\beta }(y_{i})]^{-1}(\overline{%
u^{\prime }}_{1}^{\beta },\dots ,\overline{u^{\prime }}_{i}^{\beta }))
\end{equation*}%
where $\overline{y}_{i}^{\beta }=\phi _{i}^{\beta }(y_{i})$. With these
notations, $\overline{y}_{i}^{\beta }\mapsto \Theta _{i}^{\alpha \beta }(%
\overline{y}_{i}^{\beta })=[\tilde{\Psi}_{i}^{\alpha }(y_{i})]\circ \lbrack 
\tilde{\Psi}_{i}^{\beta }(y_{i})]^{-1}$ is a conveniently smooth map.\newline

From assumption 3 and assumption 5 written over the open sets $U_{i}^{\alpha
}$ and $U_{j}^{\alpha }$ (resp. $U_{i}^{\beta }$ and $U_{j}^{\beta }$), we
have%
\begin{equation*}
\left( \tilde{\Psi}_{j}^{\alpha }\left( y_{j}\right) \right) ^{-1}\circ
\iota _{i}^{j}=\lambda _{i}^{j}\circ \left( \tilde{\Psi}_{i}^{\alpha }\left(
y_{i}\right) \right) ^{-1}
\end{equation*}

Finally we get%
\begin{equation*}
\Theta _{j}^{\alpha \beta }(\bar{y_{j}})\circ \iota _{i}^{j}=\iota
_{i}^{j}\circ \Theta _{i}^{\alpha \beta }(\bar{y_{i}})
\end{equation*}
So if $\bar{y}=\underrightarrow{\lim}{\bar{y}_i}$, from the above relation, one 
can define the transition function $\Theta ^{\alpha \beta }(\bar{y})$
as an element of the Fr\'{e}chet topological group $G(\mathbb{E})$. This ends the
proof of part 2.

\section{\label{*LinearConnectionsOnDLAnchoredBanachBundles}Linear connections on direct limit of anchored Banach bundles}

\subsection{Bundle structures on the tangent bundle to a vector bundle}

Let $M$ be a smooth Banach manifold modeled on a Banach space $\mathbb{M}$ and
let $\pi:E\rightarrow M$ be a smooth Banach vector bundle on $M$ whose typical
fiber is a Banach space $\mathbb{E}$. Let $p_{E}:TE\rightarrow E$ and
$p_{M}:TM\rightarrow M$ be the canonical projections of each tangent bundle.

There exists an atlas $\left\{  U^{\alpha},\phi^{\alpha}\right\}  _{\alpha\in
A}$ of $M$ for which $E_{|U^{\alpha}}$ is trivial; therefore we obtain a chart
$({U}_{E}^{\alpha},{\phi}_{E}^{\alpha})$ on $E$, where ${U}_{E}^{\alpha}%
=\pi^{-1}(U^{\alpha})$ and s.t. ${\phi}_{E}^{\alpha}$ is a diffeomorphism from
${U}_{E}^{\alpha}$ on $\phi^{\alpha}(U^{\alpha})\times\mathbb{E}$. We also
have a chart $({U}_{TM}^{\alpha},{\phi}_{TM}^{\alpha})$ on $TM$ where
${U}_{TM}^{\alpha}=p_{M}^{-1}(U^{\alpha})$ and ${\phi}_{TM}^{\alpha}%
=(\phi^{\alpha},Tp_{M})$.

Hence the family $\left\{  T\left(  E_{|U^{\alpha}}\right)  ,T{\phi}%
_{E}^{\alpha}\right\}  _{\alpha\in A}$ where

$T{\phi}_{E}^{\alpha}:$ $T\left(  E_{|U^{\alpha}}\right)  \longrightarrow
T\left(  \phi^{\alpha}(U^{\alpha})\times\mathbb{E}\right)  =\phi^{\alpha
}(U^{\alpha})\times\mathbb{E\times M\times E}$

is the atlas describing the canonical vector bundle structure of $\left(
TE,p_{E},E\right)  $.

Let $\left(  x,u\right)  $ be an element of $E_{x}=\pi^{-1}\left(  x\right)  $
where $x\in U^{\alpha\beta}=U^{\alpha}\cap U^{\beta}\neq \emptyset$ and let $\left(
x,u,y,v\right)  $ be an element of $T_{\left(  x,u\right)  }E$. For $\left(
x^{\alpha},u^{\alpha},y^{\alpha},v^{\alpha}\right)  =T{\phi}_{E}^{\alpha
}\left(  x,u,y,v\right)  $, we have the transition functions:%
\begin{align*}
&  \left(  T\left(  \left(  \phi^{\alpha}\times\operatorname{Id}_{\mathbb{E}%
}\right)  \circ{\phi}_{E}^{\alpha}\right)  \circ\left(  T\left(  \left(
\phi^{\beta}\times\operatorname{Id}_{\mathbb{E}}\right)  \circ{\phi}%
_{E}^{\beta}\right)  \right)  ^{-1}\right)  \left(  \left(  x^{\beta}%
,u^{\beta},y^{\beta},v^{\beta}\right)  \right) \\
&  =\left(  \phi^{\alpha\beta}\left(  x^{\beta}\right)  ,{\phi}_{E}%
^{\alpha\beta}\left(  \left(  \phi^{\beta}\right)  ^{-1}\left(  y^{\beta
}\right)  \right)  u^{\beta},d\phi^{\alpha\beta}\left(  x^{\beta}\right)
y^{\beta},\left(  d\left(  {\phi}_{E}^{\alpha\beta}\circ\left(  \phi^{\beta
}\right)  ^{-1}\right)  \left(  x^{\beta}\right)  y^{\beta}\right)  u^{\beta
}\right) \\
&  +{\phi}_{E}^{\alpha\beta}\left(  \left(  \phi^{\beta}\right)  ^{-1}\left(
x^{\beta}\right)  v^{\beta}\right)
\end{align*}
\newline
where $\phi^{\alpha\beta}=\phi^{\alpha}\circ\left(  \phi^{\beta
}\right)  ^{-1}$ and $\left(  x^{\beta},\phi_{E}^{\alpha\beta}\left(
x^{\beta}\right)  u^{\beta}\right)  =\left(  {\phi}_{E}^{\alpha}\circ\left(
{\phi}_{E}^{\beta}\right)  ^{-1}\right)  \left(  x^{\beta},u^{\beta}\right)  $
for $x^{\beta}\in\phi^{\alpha}\left(  U^{\alpha\beta}\right)$.

\bigskip So, for fixed $\left(  x^{\beta},u^{\beta}\right)$, the transition
functions are linear in $\left(  y^{\beta},v^{\beta}\right)  \in
\mathbb{M\times E}$. This describes the vector bundle structure of the tangent
bundle $\left(  TE,p_{E},E\right)  $.

On the other hand, for fixed $\left(  x^{\beta},y^{\beta}\right)  $ the
transition functions of $TE$ are also linear in $\left(  u^{\beta},v^{\beta
}\right)  \in\mathbb{E}\times\mathbb{E}$ and we get a vector bundle structure
on $\left(  TE,T\pi,TM\right)  $ which appears as the derivative of the
original one on $E$.

\subsection{Connections on a Banach bundle}

\label{banachconnection}

The kernel of $T\pi:TE\longrightarrow TM$ is denoted by $VE$ and is called the
\textit{vertical bundle} over $E$. It appears as a vector bundle over $M$. It
is well known that $VE$ can also be seen as the pull-back of the bundle
$\pi:E\longrightarrow M$ over $\pi$ as described by the following diagram:%
\[%
\begin{array}
[c]{ccc}%
E\times_{M}E\simeq\pi^{\ast}E & \overset{\widehat{\pi}}{\longrightarrow} & E\\
\downarrow &  & \downarrow\pi\\
E & \overset{\pi}{\longrightarrow} & M
\end{array}
\]

We have a canonical isomorphism $E\times_{M}E\rightarrow VE$ called the
\textit{vertical lift} $vl_{E}$ defined by%
\[
vl_{E}\left(  x,u,v\right)  =\overset{.}{\gamma}\left(  0\right)
\]

where $\gamma(t)=u+tv$. This map is fiber linear over $M$.

Let $J:VE\rightarrow TE$ be the canonical inclusion.

According to \cite{Vil} we have:

\begin{definition}
\label{D_Connection} A (non linear) connection on $E$ is a bundle morphism
$V:TE\rightarrow VE$ such that $V\circ J=\operatorname{Id}_{VE}$.
\end{definition}

The datum of a connection $V$ on $E$ is equivalent to the existence of a
decomposition $TE=HE\oplus VE$ of the Banach bundle $E$ with $HE=\ker V$.

We then have the following diagram:%
\[%
\begin{array}
[c]{ccc}%
VE & \overset{V}{\longleftarrow} & TE\\
vl_{E}\uparrow &  & \downarrow D\\
\pi^{\ast}E & \overset{\widehat{\pi}}{\longrightarrow} & E
\end{array}
\]

The bundle morphism $D=\hat{\pi}\circ vl_{E}^{-1}\circ V:TE\rightarrow E$ is
called the \textit{connection map} or \textit{connector} which is a smooth
morphism of fibrations. Note that, in each fiber $T_{(x,u)}E$, the kernel of
$D$ is exactly the subspace $H_{\left(  x,u\right)  }E$ of $HE$ in
$T_{(x,u)}E$. Therefore, the datum of $D$ is equivalent to the datum of $V$.

We then have, modulo the identification $VE\simeq\pi^{\ast}E$ via $vl_{E}$:
\[%
\begin{array}
[c]{cccc}%
D: & TE & \rightarrow & E\\
& \left(  x,u,y,v\right)  & \mapsto & \left(  x,v+\omega\left(  x,u\right)
y\right)
\end{array}
\]

where $\omega\left(  x,u\right)  \in L\left(  T_{x}M,T_{\left(  x,u\right)
}E\right)  $.

If moreover, $D$ is linear on each fiber, then the connection is called a {
}\textit{linear connection}.\newline

\bigskip Modulo the identification of $U\subset M$ and $\phi(U)\subset
\mathbb{M}$ we have the following identifications:

-- $E_{|U}\equiv U\times\mathbb{E}$

-- $TM_{|U}\equiv U\times\mathbb{M}$

-- $TE_{|\pi^{-1}(U)}\equiv(U\times\mathbb{E})\times(\mathbb{M}\times\mathbb{E})$

-- $VE_{|\pi^{-1}(U)}\equiv(U\times\mathbb{E})\times\mathbb{E}$

According to these identifications, we obtain the following characterizations
of $V$ and $D$:
\begin{align*}
V(x,u,y,v)  &  =(x,u,0,v+\omega(x,u)y)\text{ }\\
D(x,u,y,v)  &  =(x,v+\omega(x,u)y)
\end{align*}
where $\omega$ is a smooth map from $U\times\mathbb{E}$ to the space
$L(\mathbb{M},\mathbb{E})$ of bounded linear operators from $\mathbb{M}$ to
$\mathbb{E}$.

\noindent This connection is \textit{linear} if and only if $\omega$ is linear
in the second variable. In this case, the relation $\Gamma
(x)(u,y)=\omega(x,u)y$ gives rise to a smooth map $\Gamma$ from $U$ to the
space of bilinear maps $L^{2}(\mathbb{E},\mathbb{M};\mathbb{E})$ called
\textit{local Christoffel components} of the connection.

Conversely, a connection can be given by a collection $\left(  U^{\alpha
},\omega^{\alpha}\right)  $ of local maps $\omega^{\alpha}:U^{\alpha}%
\times\mathbb{E}\rightarrow L(\mathbb{M},\mathbb{E})$ on a covering
$(U^{\alpha})$ of $M$ with adequate classical conditions of compatibility
between $\left(  U^{\alpha},\omega^{\alpha}\right)  $ and $\left(  U^{\beta
},\omega^{\beta}\right)  $ where $U_{a}\cap U_{\beta}\neq\emptyset$.

\begin{remark}
\label{R_ExistsConnection} It is classical that if $M$ is smooth paracompact, then there always exists a connection on $M$ and also on each Banach bundle over $M$. However, these
assumptions impose the same assumptions on the Banach space $\mathbb{M}$.
\newline
On the other hand, it is well known that there exist linear connections on a Banach manifold without such assumptions.
For instance, if $TM\equiv M\times\mathbb{M}$ there always exists a (trivial) connection on $M$.
But there are further situations for which a linear connection exists on a Banach manifold. For example, there exist linear connections on loop spaces (see for instance \cite{CrFa}) or on the manifold $\mathcal{M}(\mu)$ of strictly positive probability densities of a probability space $(\Omega,\Sigma,\mu)$ (cf. \cite{LoQu}).
\end{remark}

\begin{definition}\label{D_Koszul}
A Koszul connection on $E$ is a $\mathbb{R}$-bilinear map
$\nabla:\mathfrak{X}(M)\times\underline{E}\rightarrow\underline{E}$ which
fulfills the following properties:

-- $\nabla_{X}\left(  f\sigma\right)  =df(X)\sigma+f\nabla_{X}\sigma$

-- $\nabla_{fX}\sigma=f\nabla_{X}\sigma$

for any function $f$ on $M$, $X\in\chi(M)$ and $\sigma\in\underline{E}$.
\end{definition}

\noindent Given  a linear connection $D$  on a Banach bundle $\pi : E \rightarrow M$,  we obtain a  covariant derivative  $\nabla:\mathfrak{X}(M)\times\underline{E}\rightarrow\underline{E}$
which is a Koszul connection.  Since any (linear)  connection induces naturally a (linear) connection  on the restriction $E_{| U}$ of $E$ to any open set $U$ of $M$, we also obtain a covariant derivative
$\nabla^U:\mathfrak{X}(U)\times \underline{E_{|U}} \rightarrow \underline{E_{|U}}$ with the correspondent previous properties
for any function $f$ on $U$, $X\in\mathfrak{X}(U)$ and $\sigma\in\underline{E_{|U}}$.\\

\noindent {\it Unfortunately, in general, a Koszul connection   may be  not localizable  in the following  sense}: \\
since any local section of $E$ (resp. any local vector field on $M$) can not be always  extended to a global section of $E$ (resp. to a global vector field on $M$), the previous operator $\nabla $ can not  always induce a (local) operator $\nabla^U$ as previously. Therefore, in this work, a {\it  Koszul connection will always be taken in the sense of the covariant derivative associated to a linear connection $D$ on $E$}. In particular, for any $x\in M$,  the value $\nabla_{X}\sigma(x)$ only depends on of the value of $X$ at $x$ and the $1$-jet of $\sigma$ at $x$.

In a local trivialization $E_{|U}\equiv U\times\mathbb{E}$, a local section
$\sigma$ of $E$, defined on $U$, can be identified with a map $\sigma
:U\rightarrow\mathbb{E}$. Then $\nabla$ has the local expression:
\[
\nabla_{X}\sigma=d\sigma(X)+\Gamma(\sigma,X)
\]
where $\Gamma$, smooth map from $U$ to $L^{2}(\mathbb{E},\mathbb{M};\mathbb{E})$, is the local Christoffel components of the connection  $D$ which will be  also called the local \textit{Christoffel components} of $\nabla$.

\begin{remark}\label{koszulreg}
If $M$ is smooth regular, then, as classically in finite dimension, any covariant  derivative $\nabla:\mathfrak{X}(M)\times \underline{E}\rightarrow\underline{E}$ which
fulfills the previous  properties (i) and (ii) is localizable. Therefore, in this case, there is a one-to-one correspondence between  such covariant derivative  and linear connection on $E$ as in the finite dimensional framework.
\end{remark}

Finally if $E_{|U}\equiv U\times\mathbb{E}$ and $E_{|U^{\prime}}\equiv
U^{\prime}\times\mathbb{E}$ are local trivializations such that $U\cap
U^{\prime}\not =\emptyset$, then we have a smooth map $g:U\cap U^{\prime
}\rightarrow GL(\mathbb{E})$ such that $\sigma_{|U^{\prime}}=g\sigma_{|U}$ for
any section defined on $U\cup U^{\prime}$. Therefore the Christoffel component
$\Gamma$ and $\Gamma^{\prime}$ of $\nabla$ on $U\cap U^{\prime}$ are linked by
the relation
\[
\Gamma^{\prime}(X,\sigma)=g^{-1}dg(X,\sigma)+g^{-1}\Gamma(X,g \sigma).
\]

\subsection{Direct limit of Banach connections}

\begin{definition}
Let $\left(  E_{n},\pi_{n},M_{n}\right)  _{n\in\mathbb{N}^{\ast}}$
be a strong ascending   sequence of Banach vector bundles where $\varepsilon_{n}%
^{n+1}:M_{n}\longrightarrow M_{n+1}$ and $\lambda_{n}^{n+1}:E_{n}%
\longrightarrow E_{n+1}$ are the compatible bonding maps. \newline A sequence
of connections $D_{n}:TE_{n}\longrightarrow E_{n}$ is called a strong ascending  sequence
of Banach connections if%
\[
\lambda_{n}^{n+1}\circ D_{n}=D_{n+1}\circ T\lambda_{n}^{n+1}.
\]

\end{definition}

\begin{theorem}
\label{T_DirectLimitConnections}Let $\left(  D_{n}\right)
_{n\in\mathbb{N}^{\ast}}$ be a strong ascending  sequence of Banach connections on an
ascending sequence $\left(  E_{n},\pi_{n},M_{n}\right)
_{n\in\mathbb{N}^{\ast}}$ of Banach bundles and assume that $(M_{n}%
)_{n\in\mathbb{N}^{\ast}}$ has the direct limit chart property at each point
of $x\in M=\underrightarrow{\lim}M_{n}$.

Then the direct limit $D=\underrightarrow{\lim}D_{n}$ is a connection on the
convenient vector bundle $\left(  \underrightarrow{\lim}E_{n},\underrightarrow
{\lim}\pi_{n},\underrightarrow{\lim}M_{n}\right)$.
\end{theorem}

\begin{proof}
Let $x$ be in $\underrightarrow{\lim}M_{n}$. We suppose that $x\in M_{n_{0}}$.
According to Definition  \ref{D_AscendingSequenceBanachVectorBundles}, let
 $\left(  U^{\alpha}%
,\phi^{\alpha}\right)  $ be a chart of $M=\varinjlim M_{n}$ around $x$ which satisfies the assumption (4) and (5), where
$\phi^{\alpha}:U^{\alpha}=\bigcup\limits_{i\geq n_{0}}U_{i}^{\alpha
}\longrightarrow O^{\alpha}=\bigcup\limits_{i\geq n_{0}}O_{i}^{\alpha}$ with
$\left(  U_{i}^{\alpha},\phi_{i}^{\alpha}\right)  $ is a chart around $x_{i}$
and $O_{i}^{\alpha}=\phi_{i}^{\alpha}\left(  U_{i}^{\alpha}\right)
\subset\mathbb{M}_{i}$. Moreover, ${E_i}_{| U_i}$ is trivial.\newline Denote by $D_{i}^{\alpha}$ the expression of
the connection $D_{i}$ in local charts. We then have $D_{i}^{\alpha}%
(x_{i}^{\alpha},u_{i}^{\alpha},y_{i}^{\alpha},v_{i}^{\alpha})=(x_{i}^{\alpha
},v_{i}^{\alpha}+\omega_{i}^{\alpha}(x_{i}^{\alpha},u_{i}^{\alpha}%
)y_{i}^{\alpha})$ where $\omega_{i}^{\alpha}$ is a smooth map from
$O_{i}^{\alpha}\times\mathbb{E}_{i}$ to the space $L(\mathbb{M}_{i}%
,\mathbb{E}_{i})$ of bounded linear operators from $\mathbb{M}_{i}$ to
$\mathbb{E}_{i}$.

Using the relations {$\lambda_{i}^{i+1}\circ D_{i}$}$=${$D_{i+1}\circ T\lambda
_{i}^{i+1}$} we have the following diagram:
%{\scalefont{0.65}\selectfont
\[%
\begin{array}[c]{ccccccc}%
 O_{i}^{\alpha}\times\mathbb{E}_{i}\times\mathbb{M}_{i}\times\mathbb{E}_{i} &
\overset{T{\phi}_{E_{i}}^{\alpha}}{\longleftarrow} &  T\left(
E_{i|U_{i}^{\alpha}}\right)   & \overset{T\lambda_{i}^{i+1}}{\longrightarrow}
& T\left(  E_{i+1|U_{i+1}^{\alpha}}\right)   &
\overset{T{\phi}_{E_{i+1}}^{\alpha}}{\longrightarrow} & O_{i+1}^{\alpha}%
\times\mathbb{E}_{i+1}\times\mathbb{M}_{i+1}\times\mathbb{E}_{i+1}\\
D_{i}^{\alpha}\downarrow &  & D_{i} \downarrow &  &
\downarrow D_{i+1} &  & \downarrow D_{i+1}^{\alpha}\\
O_{i}^{\alpha}\times\mathbb{E}_{i} & \overset{{\phi}_{E_{i}}^{\alpha}%
}{\longleftarrow} & E_{i|U_{i}^{\alpha}} & \overset{\lambda_{i}%
^{i+1}}{\longrightarrow} & E_{i+1|U_{i+1}^{\alpha}} & \overset{{\phi
}_{E_{i+1}}^{\alpha}}{\longrightarrow} & O_{i+1}^{\alpha}\times\mathbb{E}%
_{i+1}%
\end{array}
\]
%}

Using the expression in local coordinates $\widehat{\varepsilon_{i}^{i+1}%
}^{\alpha}:$ $O_{i}^{\alpha}\longrightarrow O_{i+1}^{\alpha}$ and the map
$\widehat{\lambda_{i}^{i+1}}:\mathbb{E}_{i}\longrightarrow\mathbb{E}_{i+1}$ we
then obtain that $\left(  D_{i}^{\alpha}\right)  _{i\geq n}$ can be realized
as a direct limit because we have:%
\[%
\begin{array}
[c]{ll}%
\left(  \widehat{\varepsilon_{i}^{i+1}}^{\alpha}\times\widehat{\lambda
_{i}^{i+1}}\right)  \circ D_{i}^{\alpha} & =\left(  \widehat{\varepsilon
_{i}^{i+1}}^{\alpha}\times\widehat{\lambda_{i}^{i+1}}\right)  \circ\left(
{\phi}_{E_{i}}^{\alpha}\circ D_{i}\circ\left(  T{\phi}_{E_{i}}^{\alpha
}\right)  ^{-1}\right)  \\
& ={\phi}_{E_{i+1}}^{\alpha}\circ\underline{\lambda_{i}^{i+1}\circ D_{i}}%
\circ\left(  T{\phi}_{E_{i}}^{\alpha}\right)  ^{-1}\\
& ={\phi}_{E_{i+1}}^{\alpha}\circ\underline{D_{i+1}\circ T\lambda_{i}^{i+1}%
}\circ\left(  T{\phi}_{E_{i}}^{\alpha}\right)  ^{-1}\\
& ={\phi}_{E_{i+1}}^{\alpha}\circ\left(  {\phi}_{E_{i+1}}^{\alpha}\right)
^{-1}\circ D_{i+1}^{\alpha}\circ T{\phi}_{E_{i+1}}^{\alpha}\circ T\lambda
_{i}^{i+1}\circ\left(  T{\phi}_{E_{i}}^{\alpha}\right)  ^{-1}\\
& =D_{i+1}^{\alpha}\circ\left(  \widehat{\varepsilon_{i}^{i+1}}^{\alpha}%
\times\widehat{\lambda_{i}^{i+1}}\times\widehat{\varepsilon_{i}^{i+1}}%
\times\widehat{\lambda_{i}^{i+1}}\right)
\end{array}
\]

We obtain an analogous result for the smooth Banach local forms\\ $\omega
_{i}^{\alpha}:U_{i}^{\alpha}\times\mathbb{E}_{i}\longrightarrow L(\mathbb{M}%
_{i},\mathbb{E}_{i})$.
\end{proof}

Using the intrinsic link between a connection $D$ and a Koszul connection
$\nabla$ we get the following result:

\begin{corollary}
\label{C_DLKoszulConnections}Let $\left(  D_{n}\right)  _{n\in\mathbb{N}%
^{\ast}}$ be a strong ascending  sequence of Banach connections on a direct sequence
$\left(  E_{n},\pi_{n},M_{n}\right)  _{n\in\mathbb{N}^{\ast}}$ of
Banach bundles and consider the associated Koszul connections $\left(
\nabla_{n}\right)  _{n\in\mathbb{N}^{\ast}}$. %Moreover assume that
%$(M_{n})_{n\in\mathbb{N}^{\ast}}$ has the direct limit chart property at each
%point of $x\in M=\underrightarrow{\lim}M_{n}.$
\newline
The direct limit $\nabla=\underrightarrow{\lim}\nabla_{n}$ is a Koszul connection on the
convenient vector bundle $\left(  \underrightarrow{\lim}E_{n},\underrightarrow
{\lim}\pi_{n},\underrightarrow{\lim}M_{n}\right)$.
\end{corollary}

\begin{example}
\label{Ex_DL_EllplocRn}Denote by $L_{loc}^{p}\left(  \mathbb{R}^{m}\right)  $
the space of locally $L^{p}$ functions on $\mathbb{R}^{m}$ ($1\leq p<+\infty
$). A function belongs to $L_{loc}^{p}\left(  \mathbb{R}^{m}\right)  $ if and
only if its restriction to any compact set $K$ of $\mathbb{R}^{m}$ belongs to
$L^{p}\left(  K\right)  $. Since $\mathbb{R}^{m}$ is an ascending sequence of
compact sets $K_{n}^{m}$ (where $K_{n}^{m}\subset\overset{\circ}{K_{n+1}%
^{m}\text{)}}$, we have%
\[
L_{loc}^{p}\left(  \mathbb{R}^{m}\right)  =\underrightarrow{\lim}L^{p}\left(
K_{n}^{m}\right).
\]

Moreover, the closure $\widehat{K_{n}^{m}}$ of the open set $K_{n+1}%
^{m}\backslash K_{n}^{m}$ is also compact and we have $L^{p}\left(
K_{n+1}^{m}\right)  =L^{p}\left(  K_{n}^{m}\right)  \oplus L^{p}\left(
\widehat{K_{n}^{m}}\right)  $. Therefore the sequence of Banach spaces
$\left(  L^{p}\left(  K_{n}^{m}\right)  \right)  _{n\in\mathbb{N}^{\ast}}$ is
an ascending sequence of complemented Banach spaces. Since the tangent bundle
to each $L^{p}\left(  K_{n}^{m}\right)  $ is trivial, there exists a (trivial)
Koszul connection on this Banach bundle. Therefore we get a Koszul connection on $L_{loc}^{p}\left(  \mathbb{R}^{m}\right)$.
\end{example}

\begin{example}
\label{Ex_DLCartesianProductBanachLieGroups}Let $\left(  H_{n}\right)
_{n\in\mathbb{N}^{\ast}}$ be a sequence of Banach Lie groups and consider the
Banach Lie group of cartesian products $G_{n}=%
%TCIMACRO{\dprod \limits_{k=1}^{n}}%
%BeginExpansion
{\displaystyle\prod\limits_{k=1}^{n}}
%EndExpansion
H_{k}.$ The weak direct product $%
%TCIMACRO{\dprod \nolimits_{k\in N^{\ast}}^{\ast}}%
%BeginExpansion
{\displaystyle\prod\nolimits_{k\in N^{\ast}}^{\ast}}
%EndExpansion
H_{k}$ is the set of all sequences $\left(  h_{n}\right)  _{n\in
\mathbb{N}^{\ast}}$ such that $h_{n}=1$ for all but finitely many $n$. The
weak direct product is a topological group for the box topology (see
\cite{Glo3}, 4.). In fact, this weak direct product has a structure of Lie
group modeled on the locally convex topological space $%
%TCIMACRO{\dbigoplus \limits_{k\in\mathbb{N}^{\ast}}}%
%BeginExpansion
{\displaystyle\bigoplus\limits_{k\in\mathbb{N}^{\ast}}}
%EndExpansion
\mathfrak{H}_{k}$ where $\mathfrak{H}_{k}$ is the Lie algebra of $H_{k}$. The
tangent space $TG_{n}$ is the vector bundle $G_{n}\times%
%TCIMACRO{\dbigoplus \limits_{k=1}^{n}}%
%BeginExpansion
{\displaystyle\bigoplus\limits_{k=1}^{n}}
%EndExpansion
\mathfrak{H}_{k}$. Moreover, $TG_{n}$ is a complemented subbundle of $TG_{n+1}$
and is naturally endowed with the (trivial) Koszul connection.
\end{example}

\subsection{Sprays on an anchored Banach bundle}

We begin this subsection with a brief presentation of the theory of semi-sprays on a Banach anchored bundle according to \cite{Ana}.

Let $\pi:E\rightarrow M$ be a Banach vector bundle on a Banach manifold
modeled on a Banach space $\mathbb{M}$ whose fiber is modeled on a Banach
space $\mathbb{E}$.

\begin{definition}
\label{D_AnchoredBanachBundle}A morphism of vector bundles $\rho:E\rightarrow TM$ is
called an \textit{anchor}. $\left(E,\pi,M,\rho\right)$ is then called a Banach
anchored bundle.
\end{definition}

\begin{definition}
\label{D_SemiSpray} A semi-spray on an anchored bundle is a vector field $S$
on $E$ such that $T\pi\circ S=\rho$.
\end{definition}

This means that, in a local trivialization $E_{|U}\equiv U\times\mathbb{E}$, we
have $T\pi(S(x,u))=\rho(x)u$ for all $(x,u)\in E_{|U}$.\newline A smooth curve
$c:I\subset\mathbb{R}\rightarrow E$ is called \textit{admissible} if the
tangent vector $\gamma^{\prime}(t)$ of $\gamma=\pi\circ c$ is precisely
$\rho(c(t))$.\\

From \cite {Ana}, we have the following characterization of a semi-spray:
\begin{theorem}
A vector field $S$ on $E$ is a semi-spray if and only if each integral curve
of $S$ is an admissible curve.
\end{theorem}

In a local trivialization $E_{|U}\equiv U\times\mathbf{E}$, a semi-spray can be
written as

\[
S(x,u)=\left(  x,u,\rho(x)u,-2G(x,u)\right).
\]

The \textit{Euler field} $C$ is the global vector field on $E$ which is
tangent to the fiber of $\pi$ (\textit{i.e. vertical}) and such that the flow of
$C$ is an infinitesimal homothety on each fiber. A semi-spray $S$ is called a
\textit{spray} if $S$ is invariant by the flow of $C$. This condition is
equivalent to the nullity of the Lie bracket $[C,S]$. In this case, in a local
trivialization, the function $G$ in Definition \ref{D_SemiSpray} is linear in
the second variable.\newline Conversely, a spray can be given by a collection
$(U^{\alpha},G^{\alpha})$ of local maps $G^{\alpha}:U^{\alpha}\times
\mathbb{E}\rightarrow L(\mathbb{E},\mathbb{E})$ on a covering $U^{\alpha}$ of
$M$ with adequate classical conditions of compatibility between $(U^{\alpha
},G^{\alpha})$ and $(U^{\beta},G^{\beta})$ when $U^{\alpha}\cap U^{\beta}%
\neq\emptyset$ (cf. \cite{Ana}).

Given a Koszul connection $\nabla$ on $E$ and an admissible curve
$c:I\rightarrow E$ as in the infinite dimensional case, we associate an
operator of differentiation $\nabla^{c}$ of the set of sections of $E$ along
$\gamma=\pi\circ c$ given by $\nabla^{c}\sigma=\nabla_{\dot{\gamma}}\sigma$.
In particular $c$ is a section along $\gamma$.

\begin{definition}
An admissible curve $c$ is called a geodesic of $\nabla$ if $\nabla
_{\dot{\gamma}}c=\nabla^{c}c\equiv0$.
\end{definition}

In a local trivialization $E_{|U}\equiv V\times\mathbf{E}$, an admissible curve
$c:I\rightarrow E_{|U}$ is a geodesic of $\nabla$ if and only $\pi\circ c$ is
a solution of the following differential equation:%

\[%
\begin{cases}
\dot{x}=\rho(x)u\\
\dot{u}=\Gamma(x)(u,\dot{x})
\end{cases}
\]
where $\Gamma$ is the local Christoffel component of $\nabla$ on $E_{|U}$.
Therefore, if we set $G(x,u)=-\dfrac{1}{2}\Gamma(x)(u,\rho(x)u)$, we get a
vector field $S_{U}$ on $E_{|U}$ which satisfies the relation given in
Definition \ref{D_SemiSpray} and so is a spray on $E_{|U}$. Now, according to
the compatibility conditions between the local Christoffel components, we
obtain a unique global spray associated to $\nabla$. Conversely, as in the
case of $E=TM$ (cf. \cite{Vil}), given a spray $S$ on $E$, we can associate a
unique connection $\nabla$ whose associated spray is $S$.\newline

Taking into account the classical theorem of existence of a local flow of a
vector field on a Banach manifold, we obtain:

\begin{theorem}
\label{T_SprayKoszul}Let $(E,M,\rho)$ be a Banach anchored bundle. There exists a
spray on $E$ if and only there exists a Koszul connection $\nabla$ on $E$.
Moreover, there exists a canonical correspondence one-to-one between sprays and
Koszul connections on $E$ so that an admissible curve is a geodesic of the
Koszul connection $\nabla$ if and only if this curve is an integral curve of
the unique $S$ associated to $\nabla$.
\newline
When $E=TM$, there exists an
exponential map $\operatorname{Exp}:\mathcal{U}\subset TM\rightarrow M$,
defined on an open neighborhood $\mathcal{U}$ of the zero section, such that
${p_{M}}_{|\mathcal{U}}:{\mathcal U}\rightarrow M$ is a fibration whose  each fiber
$\mathcal{U}_{x}$ is a star-shaped open neighborhood of $0$ in $T_{x}M$.
Moreover, the differential of the restriction $\operatorname{Exp}_{x}$ of
$\operatorname{Exp}$ to $\mathcal{U}_{x}$ is equal to $Id_{T_{x}M}$ at $0$. In
particular, $\operatorname{Exp}_{x}$ is a diffeomorphism of a star-shaped open
neighborhood of $0\in T_{x}M$ onto an open neighborhood of $x\in M$.
\end{theorem}

\section{\label{*DL_SequencesAlmostBanachLieAlgebroids}Direct limits of
sequences of almost Banach Lie algebroids}

\subsection{\label{**Algebroid}Almost Banach Lie algebroids}

Let $\left(E,\pi,M,\rho\right)$ be a Banach anchored bundle.

If $\underline{E}$ denotes the $C^\infty(M)$-module of smooth sections of $E$,  the morphism $\rho$ gives rise to a $C^\infty(M)$-module morphism \underline{$\rho$
}$:\underline{E}\rightarrow$\underline{$TM$}=$\mathfrak{X}(M)$ defined for
every $x$ $\in$ $M$ and every section $s$ of $E$ by: $\left(  \underline{\rho
}\left(  s\right)  \right)  \left(  x\right)  =\rho\left(  s\left(  x\right)
\right)  $ and still denoted by $\rho$.

\begin{definition}
\label{D_AlmostLieBracket}An \textit{almost Lie bracket} on an anchored bundle
$(E,\pi,M,\rho)$ is a bilinear map $[\;,\;]_{E}:\underline{E}\times\underline
{E}:\longrightarrow\underline{E}$ which satisfies the following properties:

1. $[\;,\;]_{E}$ is antisymmetric;

2. Leibniz property :
\[
\forall s_{1},s_{2}\in\underline{E},\forall f\in C^{\infty}\left(  M\right)
,\ [s_{1},fs_{2}]_{E}=f.[s_{1},s_{2}]+df(\rho(s_{1})).s_{2}.%
\]
\end{definition}

\begin{definition}
\label{D_LieBracket}A Lie bracket\textbf{ } is an almost Lie bracket whose
jacobiator vanishes:
\[
\forall s_{1},s_{2},s_{3}\in\underline{E},[s_{1},[[s_{2},s_{3}]]+[s_{2}%
,[[s_{3},s_{1}]]+[s_{3},[[s_{1},s_{2}]]=0
\]

\end{definition}

\begin{definition}
\label{D_AlmostLieAlgebroid}An almost Banach Lie algebroid is an anchored
bundle $(E,\pi,M,\rho)$ provided with an almost Lie bracket $[\;,\;]_{E}$. When
$[\;,\;]_{E}$ is in fact a Lie bracket the associated structure $(E,M,\rho
,[.,.]_{E})$ is called a Banach Lie algebroid.
\end{definition}

If $(E,\pi,M,\rho,[\;,\;_{E})$ is a Banach Lie algebroid, $\rho:\underline
{E}\rightarrow\mathfrak{X}(M)$ is a Lie algebra morphism; in particular, we
have $[\rho s_{1},\rho s_{2}]=\rho\left(  \lbrack s_{1},s_{2}]_{E}\right)  $.
\newline

Notice that the converse is not true in general (take $\rho\equiv0$ for instance).

\begin{definition}
\label{D_Algebroid}When we have $[\rho s_{1},\rho s_{2}]=\rho([s_{1}%
,s_{2}]_{E})$ for all sections $s_{1},s_{2}\in\underline{E}$, we will say that
\textit{$\rho$ is a Lie morphism}. In this case $(E,\pi,M,\rho),[\;,\;]_{E})$ is
called an \textbf{  algebroid}.
\end{definition}

In general the almost Lie bracket of an algebroid $(E,\pi,M,\rho,[\;,\;])$ does
not satisfy the Jacobi identity. \newline

\begin{remark}
\label{R_TermilologyAlmostLieAlgebroids}Since the terminology of almost
Poisson bracket seems generally adopted in the most recent papers on
nonholonomic mechanics, in this work we have adopted the definition of an
almost Lie algebroid given in \cite{LMM}. Therefore taking into account the
relation between almost Linear Poisson bracket and almost Lie bracket, this
terminology seems to us well adapted. Therefore and according to \cite{PoPo},
we use the denomination "algebroid" for an almost algebroid such that the
anchor is a morphism of Lie algebras. Note that in \cite{PoPo} or in
\cite{GrJo} an almost Lie algebroid corresponds to the previous definition of
an algebroid and our denomination "almost algebroid " corresponds to
"quasi-Lie algebroid" in \cite{PoPo} or in \cite{GrJo}$.$
\end{remark}

%On  $(E,M,\rho)$ an anchored bundle  $({\cal A},M,\rho)$, there always exists an almost Lie bracket:

\begin{example}
\label{Ex_StructureLieAlgebroidLinkedActionBanachLieGroup}Consider a smooth
right action $\psi:M\times G\longrightarrow M$ of a connected Lie group $G$
on a Banach manifold $M$. Denote by $\mathcal{G}$ the Lie algebra of $G$. We
then have a natural morphism $\xi$ of Lie algebras from $\mathcal{G}$ to
$\mathfrak{X}(M)$ defined by:%
\[
\xi_{X}\left(  x\right)  =T_{\left(  x,e\right)  }\psi\left(  0,X\right)  .
\]

For any $X$ and $Y$ in $\mathcal{G}$, we have: $\xi_{\left\{  X,Y\right\}
}=\left[  \xi_{X},\xi_{Y}\right]  $ where $\left\{  .,.\right\}  $ denotes the
bracket on the Lie algebra $\mathcal{G}$ (see for instance \cite{KrMi}, 36.12).

On the trivial bundle $M\times\mathcal{G}$, each section can be identified
with a map $\sigma:M\longrightarrow\mathcal{G}.$ We then define a Lie bracket
$\{\{.,.\}\}$ on the set of such sections by:%
\[
\{\{\sigma_{1},\sigma_{2}\}\}\left(  x\right)  =\left\{  \sigma_{1}\left(
x\right)  ,\sigma_{2}\left(  x\right)  \right\}  +d\sigma_{1}\left(
\xi_{\sigma_{2}\left(  x\right)  }\right)  -d\sigma_{2}\left(  \xi_{\sigma
_{1}\left(  x\right)  }\right)  \text{.}%
\]

An anchor $\Psi:M\times\mathcal{G}\longrightarrow TM$ is defined by
$\Psi\left(  x,X\right)  =\xi_{X}\left(  x\right)  $.
\newline
Then $\left(
M\times\mathcal{G},\textrm{pr}_1,M,\Psi,\{\{.,.\}\}\right)  $ is a Banach Lie algebroid.

Moreover, if we denote by $G_{x}$ the closed subgroup of isotropy of a point
$x\in M$ and by $\mathcal{G}_{x}\subset\mathcal{G}$ its Lie subalgebra, we
have $\ker\Psi_{x}=\mathcal{G}_{x}$. If $\mathcal{G}_{x}$ is complemented in $\mathcal{G}$ for any $x\in M$ and $\rho$ has closed range, then the weak distribution $\mathcal{D}%
=\Psi\left(  M\times\mathcal{G}\right)$ is integrable and the leaf through
$x$ is its orbit $\psi\left(  x,G\right)$ (cf. \cite{Pel}, example 4.3, 3.).
\end{example}

Note that in finite dimension it is classical that a
Lie  bracket  $[\;,\;]_{E}$  on an anchored bundle $(E,\pi,M,\rho)$ 
respects the sheaf of sections of $\pi:E\rightarrow M$ or, for short,  is {\bf localizable}  (see for instance \cite{Mar}),  if  the following properties are satisfied:

\begin{enumerate}
\item[(i)]  for any open set $U$ of $M$, there exists a unique  bracket $[\;,\;]_U$ on the space of sections  $\underline{E}_{| U})$ such that, for any $s_1$ and $s_2$ in  $\underline{E}_{| U})$, we have:
$$[{s_1}_{|U},{s_1}_{|U}]_U=([s_1,s_2]_{E})_{| U}$$
\item[(ii)]  (compatibility with restriction) if $V\subset U$ are open sets, then, $[.,.]_U$  induces a unique  Lie bracket  $[.,.]_{UV}$ on $\underline{E}_{| V})$ which coincides with  $[.,.]_V$ (induced by $[.,.]_E$).
\end{enumerate}

By the same arguments as in finite dimension, when $M$ is smooth regular  any Lie bracket  $[\;,\;]_{E}$  on an anchored bundle $(E,\pi,M,\rho)$ is localizable
(cf. \cite{Pel}).

But, in general, for analog reasons as for Koszul connection, we can not prove that any Lie bracket is localizable. Unfortunately in the Banach framework, we have {\bf no example of  Lie algebroid} for which  is not localizable. Therefore at least  for finding  conditions under which a Banach Lie algebroid is integrable this condition is necessary. This condition of localization implies also that a bracket depends on the one jets of sections.   Therefore, in the sequel, we will assume that {\bf  all almost Lie bracket $[.,;]_E$ are localizable.}

\begin{remark}
\label{R_ALBracketAssociatedToKoszulConnection}If there exists a Koszul
connection ${\nabla}$ on $E$, then we get an almost Lie bracket
$[.,.]_{\nabla}$ defined by%
\[
\lbrack s_{1},s_{2}]_{\nabla}=\nabla_{\rho s_{1}}s_{2}-\nabla_{\rho s_{2}%
}s_{1}.
\]
Note that since a $\nabla$ is localizable, so is $[\;,\;]_\nabla$
\end{remark}

\bigskip

When $(E,\pi,M,\rho,[\;,\;])_{E}$ is an almost Banach Lie algebroid we can define
the following operators:

\begin{enumerate}
 \item[(i)] \textit{Lie derivative} $L_{s}^{\rho}$ according to a section $s$ of $E$:

for a smooth function $f\in\Omega^{0}\left(  M,E\right)  =\mathcal{F,}$
\[
L_{s}^{\rho}\left(  f\right)  =L_{\rho\circ s}\left(  f\right)  =i_{\rho\circ
s}\left(  df\right);  \newline%
\]

for a $q$--form $\omega\in\Omega^{q}\left(  M,E\right)  $ (where $q>0$)
\begin{align}
\left(  L_{s}^{\rho}\omega\right)  \left(  s_{1},\dots,s_{q}\right)   &
=L_{s}^{\rho}\left(  \omega\left(  s_{1},\dots,s_{q}\right)  \right)
\nonumber\\
&  -{\sum\limits_{i=1}^{q}}\omega\left(  s_{1},\dots,s_{i-1},\left[
s,s_{i}\right]  _{E},s_{i+1},\dots,s_{q}\right).
\end{align}

\item[(ii)] $\Omega\left(  M,E\right)$-\textit{value derivative} according to a
section\textit{ }$s$ of $E$:

for a smooth function $f\in\Omega^{0}\left(  M,E\right)  =\mathcal{F}$
\[
d_{\rho}f=t_{\rho}\circ df;
\]

for a $q$--form $\omega\in\Omega^{q}\left(  M,E\right)  $ (where $q>0$)
\begin{align*}
\left(  d_{\rho}\omega\right)  \left(  s_{0},\dots,s_{q}\right)   &
={\sum\limits_{i=0}^{q}}\left(  -1\right)  ^{i}L_{s_{i}}^{\rho}\left(
\omega\left(  s_{0},\dots,\widehat{s_{i}},\dots,s_{q}\right)  \right) \\
&  +{\sum\limits_{0\leq i<j\leq q}^{q}}\left(  -1\right)  ^{i+j}\left(
\omega\left(  \left[  s_{i},s_{j}\right]  _{E},s_{0},\dots,\widehat{s_{i}
},\dots,\widehat{s_{j}},\dots,s_{q}\right)  \right).
\end{align*}
\end{enumerate}

In general, we have $d_{\rho}\circ d_{\rho}\not =0$. However, $(E,M,\rho
,[\;,\;])_{E}$ is a Banach Lie algebroid if and only if $d_{\rho}\circ
d_{\rho}=0$.

\begin{definition}
Let $\psi:E\rightarrow E^{\prime}$ be a linear bundle morphism over
$f:M\rightarrow M^{\prime}$.

\begin{enumerate}
\item[(i)] A section $s^{\prime}$ of $E^{\prime}\rightarrow M^{\prime}$ and a
section $s$ of $E\rightarrow M$ are $\psi$-related if $s^{\prime}\circ
f=\psi\circ s$.

\item[(ii)] $\psi$ is a morphism of almost Banach Lie algebroids from $\left(
E,\pi,M,\rho,[\;,\;]_{E}\right)  $ to $\left(  E^{\prime},\pi^{\prime
},M^{\prime} ,\rho^{\prime},[\;,\;]_{E^{\prime}}\right)  $ if:

\begin{enumerate}
\item[(a)] $\rho^{\prime}\circ\psi=Tf\circ\rho$;

\item[(b)] for any pair of $\psi$-related sections $s_{i}^{\prime}$ and $s_{i}%
$ ($i=1,2$), we have:
\newline
$\psi([s_{1},s_{2}])=[s^{\prime}_{1},s^{\prime
}_{2}]^{\prime}\circ f$, i.e. the Lie bracket $[s^{\prime}%
_{1},s^{\prime}_{2}]^{\prime}$ and $[s_{1},s_{2}]$ are $\psi$-related.
\end{enumerate}
\end{enumerate}
\end{definition}

In a dual way, a morphism $\psi:E\rightarrow E^{\prime}$ which satisfies
property (a) is an almost Banach Lie algebroid morphism if the mapping
$\psi^{\ast}:\Omega^{q}\left(  M,E^{\prime}\right)  \rightarrow\Omega
^{q}\left(  M,E\right)  $ defined by:
\[
\left(  \psi^{\ast}\alpha^{\prime}\right)  _{x}\left(  s_{1},\dots
,s_{q}\right)  =\alpha_{f\left(  x\right)  }^{\prime}\left(  \psi\circ
s_{1},\dots,\psi\circ s_{q}\right)
\]
commutes with the differentials:
\[
d_{\rho}\circ\psi^{\ast}=\psi^{\ast}\circ d_{\rho^{\prime}}.
\]

Notice that an almost Banach Lie algebroid $\left(  E,\pi,M,\rho
,[\;,\;]_{E}\right)  $ is a Banach algebroid if and only if the anchor $\rho$
is a morphism of Banach Lie algebroids from $\left(  E,\pi,M,\rho,[\;,\;]_{E}%
\right)  $ to the canonical Banach Lie algebroid $(TM, p_{M},M,Id_{TM},[\;,\;])$.\newline

\subsection{Direct limit of almost Banach Lie algebroids}

As in the Banach framework, if $\pi:E\longrightarrow  M$ is a convenient bundle over a n.n.H. convenient manifold $M$, then we can define \textit{the convenient algebroid or Lie algebroid structure}\footnote{In this case $E$ has a structure of n.n.H. convenient manifold.} $(E,\pi,M,\rho,[\;,\;]_E)$ in an obvious way. Now coming back to the context of sequence of Banach anchored bundles, we have:
%\begin{definition}\label{D_convenientalgebroid}

\begin{definition}\label{D_DirectSequenceBanachLieAlgebroids}
\begin{enumerate}
\item[(i)] A sequence  $\left(  E_{n},\pi_{n},M_{n},\rho_{n}\right)  _{n\in\mathbb{N}^{\ast}}$ is
called a \textit{strong ascending  sequence of anchored Banach bundles} if

(1)  $\left(  E_{n},\pi_{n},M_{n},\right)  _{n\in\mathbb{N}^{\ast}}$ is a direct sequence of Banach bundles;

%$\left(  E_{n},\lambda_{n}^{m}\right)  _{n\in\mathbb{N}^{\ast},\ m\in
%\mathbb{N}^{\ast},\ n\leq m}$ is a  direct  sequence of Banach vector bundles
%$\left(  \pi_{n}:E_{n}\rightarrow M_{n}\right)  _{n\in\mathbb{N}^{\ast}}$ over
%the direct sequence of Banach manifolds $\left(  \left(  M_{n},\varepsilon
%_{n}^{m}\right)  \right)  _{n\in\mathbb{N}^{\ast},\ m\in\mathbb{N}^{\ast
%},\ n\leq m}$\newline(2) for all $n,m\in\mathbb{N}^{\ast}$ such that $n\leq
%m$,
(2)  For all $n\leq m$, we have
\[
\rho_{m}\circ\lambda_{n}^{m}=T\varepsilon_{n}^{m}\circ\rho_{n}.
\]
where $\lambda_n^m:E_n\longrightarrow E_m$ and $\epsilon_n^m:M_n\longrightarrow M_m$  are the bonding morphisms.

\item[(ii)] A sequence $\left(  E_{n},\pi_{n},M_{n},\rho_{n},[\;,\;]_{n}\right)
_{n\in\mathbb{N}^{\ast}}$ is called a    \textit{strong ascending  sequence of almost Banach
Lie algebroids} if $\left(  E_{n},\pi_{n},M_{n},\rho_{n}\right)
_{n\in\mathbb{N}^{\ast}}$ is a strong ascending   sequence of anchored Banach bundles with
the additional property:
\newline
 $\lambda_{n}^{m}:E_{n}\rightarrow E_{m}$
is an almost Banach algebroid morphism between the almost Banach Lie
algebroids $\left(  E_{n},\pi_{n},M_{n},\rho_{n},[\;,\;]_{n}\right)  $ and
$\left(  E_{m},\pi_{m},M_{m},\rho_{m},[\;,\;]_{m}\right)  .$
\end{enumerate}
\end{definition}

\begin{theorem}
\label{T_DLBanachLieAlgebroids_ConvenientLieAlgebroids} 1. If $\left(  E_{n},\pi_{n},M_{n},\rho_{n})\right)  _{n\in\mathbb{N}^{\ast}}$ is a
strong ascending   sequence of anchored bundles, then $\left(  \underrightarrow
{\lim}E_{n},\underrightarrow{\lim}\pi_{n},\underrightarrow{\lim}%
M_{n},\underrightarrow{\lim}\rho_{n})\right)  $ is a convenient anchored
bundle. \newline Moreover, $\left(  \underrightarrow{\lim}E_{n}%
,\underrightarrow{\lim}\pi_{n},\underrightarrow{\lim}M_{n},\underrightarrow
{\lim}\rho_{n},\underrightarrow{\lim}[\;,\;]_{n}\right)  $ is a convenient
algebroid (resp. a convenient Lie algebroid) if each $\left(  E_{n},\pi
_{n},M_{n},\rho_{n},[\;,\;]_{n}\right)  $ is a Banach algebroid (resp. a
Banach Lie algebroid) for {$n$}${\in\mathbb{N}^{\ast}}$.
\newline
2. If $\left(  E_{n},\pi_{n},M_{n},\rho_{n},[\;,\;]_{n}\right)  _{n\in
\mathbb{N}^{\ast}}$ is a strong ascending  sequence of almost Banach Lie algebroids, then
$\left(  \underrightarrow{\lim}E_{n},\underrightarrow{\lim}\pi_{n}%
,\underrightarrow{\lim}M_{n},\underrightarrow{\lim}\rho_{n},\underrightarrow
{\lim}[\;,\;]_{n}\right)  $ is an almost convenient Lie algebroid. \newline
Moreover, $\left(  \underrightarrow{\lim}E_{n},\underrightarrow{\lim}\pi
_{n},\underrightarrow{\lim}M_{n},\underrightarrow{\lim}\rho_{n}%
,\underrightarrow{\lim}[\;,\;]_{n}\right)  $ is a convenient algebroid (resp.
a convenient Lie algebroid) if each $\left(  E_{n},\pi_{n},M_{n},\rho
_{n},[\;,\;]_{n}\right)  $ is a Banach algebroid (resp. a Banach Lie
algebroid) for {$n$}${\in\mathbb{N}^{\ast}}$.
\end{theorem}

\begin{proof}

1. According to Proposition \ref{P_StructureOnDirectLimitLinearBundles},
$\left(  \underrightarrow{\lim}E_{n},\underrightarrow{\lim}\pi_{n}%
,\underrightarrow{\lim}M_{n}\right)  $ can be endowed with a structure of
convenient vector bundle whose base is modeled on the LB-space
$\underrightarrow{\lim}\mathbb{M}_{n}$ and whose structural group is the Fr\'echet topological group $G(\mathbb{E})$.

2. Let $\left(  s_{n}^{1}\right)  _{n\in\mathbb{N}^{\ast}}$ and $\left(
s_{n}^{2}\right)  _{n\in\mathbb{N}^{\ast}}$ be  sequences of sections of
the linear bundles $\pi_{n}:E_{n}\rightarrow M_{n}$, i.e. fulfilling the
conditions :
\begin{equation}
\left\{
\begin{array}
[c]{c}%
\lambda_{n}^{m}\circ s_{n}^{1}=s_{m}^{1}\circ\varepsilon_{n}^{m}\\
\lambda_{n}^{m}\circ s_{n}^{2}=s_{m}^{2}\circ\varepsilon_{n}^{m}%
\end{array}
\right.  \label{_Comp_Sect}%
\end{equation}

In order to define a structure of almost convenient Lie structure on the
direct limit we have to prove the compatibility of the brackets
\begin{equation}
\lambda_{n}^{m}\circ\left[  s_{n}^{1},s_{n}^{2}\right]  _{E_{n}}=\left[
s_{m}^{1},s_{m}^{2}\right]  _{E_{m}}\circ\varepsilon_{n}^{m}
\label{_Comp_SectBrackets}%
\end{equation}

and the compatibility of the Leibniz properties:
\begin{equation}
\lambda_{n}^{m}\circ\left[  s_{n}^{1},g_{n}\times s_{n}^{2}\right]  _{E_{n}%
}=\left[  s_{m}^{1},g_{m}\times s_{m}^{2}\right]  _{E_{m}}\circ\varepsilon
_{n}^{m} \label{_Comp_SectLeibniz}%
\end{equation}

a) In order to prove (\ref{_Comp_SectBrackets}) we use the morphisms
$\lambda_{n}^{m}:E_{n}\longrightarrow E_{m}$ of Lie algebroids over
$\varepsilon_{n}^{m}:M_{n}\longrightarrow M_{m}:$
\begin{equation}
d_{\rho_{n}}\circ\left(  \lambda_{n}^{m}\right)  ^{\ast}=\left(  \lambda
_{n}^{m}\right)  ^{\ast}\circ d_{\rho_{m}} \label{_MorphismLieAlgebroids}%
\end{equation}

applied to $\alpha_{m}\in\Omega^{1}\left(  M_{m},E_{m}\right)  .$

We then have $\left(  d_{\rho_{n}}\circ\left(  \lambda_{n}^{m}\right)  ^{\ast
}\left(  \alpha_{m}\right)  \right)  \left(  s_{n}^{1},s_{n}^{2}\right)
=\left(  \left(  \lambda_{n}^{m}\right)  ^{\ast}\circ d_{\rho_{m}}\left(
\alpha_{m}\right)  \right)  \left(  s_{n}^{1},s_{n}^{2}\right)  ,$
\newline\newline For the LHS, we have:
\begin{align*}
&  \left(  d_{\rho_{n}}\circ\left(  \lambda_{n}^{m}\right)  ^{\ast}\left(
\alpha_{m}\right)  \right)  \left(  s_{n}^{1},s_{n}^{2}\right) \\
&  =L_{\rho_{n}\circ s_{n}^{1}}\left(  \left(  \left(  \lambda_{n}^{m}\right)
^{\ast}\left(  \alpha_{m}\right)  \right)  \left(  s_{n}^{2}\right)  \right)
-L_{\rho_{n}\circ s_{n}^{2}}\left(  \left(  \left(  \lambda_{n}^{m}\right)
^{\ast}\left(  \alpha_{m}\right)  \right)  \left(  s_{n}^{1}\right)  \right)
-\left(  \left(  \lambda_{n}^{m}\right)  ^{\ast}\left(  \alpha_{m}\right)
\right)  \left[  s_{n}^{1},s_{n}^{2}\right]  _{E_{n}}\\
&  =X_{m}^{1}\left(  \alpha_{m}\left(  \lambda_{n}^{m}\circ s_{n}^{2}\right)
\right)  -X_{m}^{2}\left(  \alpha_{m}\left(  \lambda_{n}^{m}\circ s_{m}%
^{1}\right)  \right)  -\alpha_{m}\left(  \lambda_{n}^{m}\circ\left[  s_{n}%
^{1},s_{n}^{2}\right]  _{E_{n}}\right)
\end{align*}

where $X_{m}^{a}=\rho_{m}\circ s_{m}^{a}$ with $a=1,2$ fulfill the relation
$X_{m}^{a}\left(  f_{m}\right)  =X_{n}^{a}\left(  f_{n}\right)  $ for
$f_{m}=\alpha_{m}\circ s_{m}$.

For the RHS, we get:
\begin{align*}
&  \left(  \left(  \lambda_{n}^{m}\right)  ^{\ast}\left(  d_{\rho_{m}}\left(
\alpha_{m}\right)  \right)  \right)  \left(  s_{n}^{1},s_{n}^{2}\right) \\
&  =d_{\rho_{m}}\left(  \alpha_{m}\right)  \left(  \lambda_{n}^{m}\circ
s_{n}^{1},\lambda_{n}^{m}\circ s_{n}^{2}\right). \\
&  =L_{\rho_{m}\circ\lambda_{n}^{m}\circ s_{n}^{1}}\left(  \alpha_{m}\left(
\lambda_{n}^{m}\circ s_{n}^{2}\right)  \right)  -L_{\rho_{m}\circ\lambda
_{n}^{m}\circ s_{n}^{2}}\left(  \alpha_{m}\left(  \lambda_{n}^{m}\circ
s_{n}^{1}\right)  \right)  -\alpha_{m}\left[  \lambda_{n}^{m}\circ s_{n}%
^{1},\lambda_{n}^{m}\circ s_{n}^{2}\right]  _{E_{m}}\\
&  =L_{\rho_{m}\circ s_{m}^{1}}\left(  \alpha_{m}\left(  \lambda_{n}^{m}\circ
s_{n}^{2}\right)  \right)  -L_{\rho_{m}\circ s_{m}^{2}}\left(  \alpha
_{m}\left(  \lambda_{n}^{m}\circ s_{n}^{1}\right)  \right)  -\alpha_{m}\left[
\lambda_{n}^{m}\circ s_{n}^{1},\lambda_{n}^{m}\circ s_{n}^{2}\right]  _{E_{m}%
}\\
&  =X_{m}^{1}\left(  \alpha_{m}\left(  \lambda_{n}^{m}\circ s_{n}^{2}\right)
\right)  -X_{m}^{2}\left(  \alpha_{m}\left(  \lambda_{n}^{m}\circ s_{m}%
^{1}\right)  \right)  -\alpha_{m}\left[  \lambda_{n}^{m}\circ s_{n}%
^{1},\lambda_{n}^{m}\circ s_{n}^{2}\right]  _{E_{m}}.%
\end{align*}

Finally, we have for all $\alpha_{m}\in\Omega^{1}\left(  M_{m},E_{m}\right)  $,
\newline$\alpha_{m}\left(  \lambda_{n}^{m}\left(  \left[  s_{n}^{1},s_{n}%
^{2}\right]  _{E_{n}}\right)  \right)  =\alpha_{m}\left[  \lambda_{n}^{m}\circ
s_{n}^{1},\lambda_{n}^{m}\circ s_{n}^{2}\right]  _{E_{m}}$ and we obtain:
$\lambda_{n}^{m}\circ\left[  s_{n}^{1},s_{n}^{2}\right]  _{E_{n}}=\left[
\lambda_{n}^{m}\circ s_{n}^{1},\lambda_{n}^{m}\circ s_{n}^{2}\right]  _{E_{m}%
}$.\newline Using $\lambda_{n}^{m}\circ s_{n}^{a}=s_{m}^{a}\circ
\varepsilon_{n}^{m}$, we have: $\lambda_{n}^{m}\circ\left[  s_{n}^{1}%
,s_{n}^{2}\right]  _{E_{n}}=\left[  s_{m}^{1},s_{m}^{2}\right]  _{E_{m}}%
\circ\varepsilon_{n}^{m}$.

b) To prove (\ref{_Comp_SectLeibniz}) we are going to establish that
\[
\lambda_{n}^{m}\circ\left(  g_{n}\times\left[  s_{n}^{1},s_{n}^{2}\right]
_{E_{n}}+\left(  \rho_{n}\left(  s_{n}^{1}\right)  \right)  \left(
g_{n}\right)  \times s_{n}^{2}\right)  =\left(  g_{m}\times\left[  s_{m}%
^{1},s_{m}^{2}\right]  _{E_{m}}+\left(  \rho_{m}\left(  s_{m}^{1}\right)
\right)  \left(  g_{m}\right)  \times s_{m}^{2}\right)  \circ\varepsilon
_{n}^{m}%
\]

We can write:
\begin{align*}
&  \lambda_{n}^{m}\circ\left(  g_{n}\times\left[  s_{n}^{1},s_{n}^{2}\right]
_{E_{n}}+\left(  \rho_{n}\left(  s_{n}^{1}\right)  \right)  \left(
g_{n}\right)  \times s_{n}^{2}\right) \\
&  =\lambda_{n}^{m}\circ\left(  g_{n}\times\left[  s_{n}^{1},s_{n}^{2}\right]
_{E_{n}}\right)  +\lambda_{n}^{m}\circ\left(  \left(  \rho_{n}\left(
s_{n}^{1}\right)  \right)  \left(  g_{n}\right)  \times s_{n}^{2}\right) \\
&  =g_{n}\times\left(  \lambda_{n}^{m}\circ\left[  s_{n}^{1},s_{n}^{2}\right]
_{E_{n}}\right)  +\lambda_{n}^{m}\left(  X_{n}^{1}\left(  g_{n}\right)
\right)  \times\lambda_{n}^{m}\circ s_{n}^{2}\quad\text{(}\lambda_{n}%
^{m}\text{ is a morphism)}\\
&  =g_{n}\times\left(  \left[  s_{m}^{1},s_{m}^{2}\right]  _{E_{m}}%
\circ\varepsilon_{n}^{m}\right)  +X_{m}^{1}\left(  g_{m}\right)
\circ\varepsilon_{n}^{m}\times s_{m}^{2}\circ\varepsilon_{n}^{m}\quad\text{cf.
(\ref{_Comp_SectBrackets})}\\
&  =\left(  g_{m}\circ\varepsilon_{n}^{m}\right)  \times\left(  \left[
s_{m}^{1},s_{m}^{2}\right]  _{E_{m}}\circ\varepsilon_{n}^{m}\right)  +\left(
X_{m}^{1}\left(  g_{m}\right)  \times s_{m}^{2}\right)  \circ\varepsilon
_{n}^{m}\\
&  =\left(  g_{m}\times\left[  s_{m}^{1},s_{m}^{2}\right]  _{E_{m}}\right)
\circ\varepsilon_{n}^{m}+\left(  \rho_{m}\left(  s_{m}^{1}\right)  \left(
g_{m}\right)  \times s_{m}^{2}\right)  \circ\varepsilon_{n}^{m}\\
&  =\left(  g_{m}\times\left[  s_{m}^{1},s_{m}^{2}\right]  _{E_{m}}+\left(
\rho_{m}\left(  s_{m}^{1}\right)  \right)  \left(  g_{m}\right)  \times
s_{m}^{2}\right)  \circ\varepsilon_{n}^{m}\text{.}%
\end{align*}
\qquad

3. Now, from the previous construction of $\underrightarrow{\lim}[\;,\;]_{n}$,
it is clear that if $\rho_{n}$ is a morphism of almost algebroids from $\left(
E_{n},\pi_{n},M_{n},\rho_{n},[\;,\;]_{n}\right)  $ to the canonical Banach Lie
algebroid $(TM_{n},\pi_n,M_{n},\textrm{Id}_{TM_n},[\;,\;])$, then $\underrightarrow{\lim}\rho_{n}$
satisfies
\[
\underrightarrow{\lim}\rho_{n}(\underrightarrow{\lim}[\;,\;]_{n}%
)=[\underrightarrow{\lim}\rho_{n}(.),\underrightarrow{\lim}\rho_{n}(.)].
\]
Moreover, it is also easy to show that if each bracket $[\;,\;]_{n}$ satisfies
the Jacobi identity, then $\underrightarrow{\lim}[\;,\;]_{n}$ satisfies also a
Jacobi identity. These last proofs are left to the reader.
\end{proof}

\begin{corollary}
\label{DLConnectionBundle} Let $\left(  D_{n}\right)  _{n\in\mathbb{N}^{\ast}%
}$ be a strong ascending    sequence of Banach connections on a strong ascending  sequence
$\left(  E_{n},\pi_{n},M_{n}\right)  _{n\in\mathbb{N}^{\ast}}$ of
Banach bundles. %and assume that $(M_{n})_{n\in\mathbb{N}^{\ast}}$ has the
%direct limit chart property at each point of $x\in M=\underrightarrow{\lim
%}M_{n}.$
 Then there exists an almost convenient Lie algebroid structure on the bundle
$\left(  \underrightarrow{\lim}E_{n},\underrightarrow{\lim}\pi_{n}%
,\underrightarrow{\lim}M_{n}\right)  $.
\end{corollary}

\begin{proof}
On each anchored bundle $(E_{n},\pi_{n},M_{n},\rho_{n})$, we denote by
$\nabla_{n}$ the $E_{n}$-Koszul connection associated to $D_{n}$. Therefore
$[s_{n}^{1},s_{n}^{2}]_{n}=\nabla_{\rho_{n}\left(  s_{n}^{1}\right)  }%
^{n}s_{n}^{2}-\nabla_{\rho_{n}\left(  s_{n}^{2}\right)  }^{n}s_{n}^{1}$
defines an almost Lie bracket on $(E_{n},\pi_{n},M_{n},\rho_{n})$. Since
$\left(  D_{n}\right)  _{n\in\mathbb{N}^{\ast}}$ is a direct sequence of
Banach connections, it follows that the sequence $\left(
[\;,\;]_{n}\right)  _{n\in\mathbb{N}^{\ast}}$ of almost brackets satisfies the
property (3) of Definition \ref{D_DirectSequenceBanachLieAlgebroids}.
Therefore, from Theorem \ref{T_DLBanachLieAlgebroids_ConvenientLieAlgebroids},
$\underrightarrow{\lim}[\;,\;]_{n}$ is an almost Lie bracket on the convenient
anchored bundle $\left(  \underrightarrow{\lim}E_{n},\underrightarrow{\lim}%
\pi_{n},\underrightarrow{\lim}M_{n},\underrightarrow{\lim}\rho_{n}\right)$.
\end{proof}

\section{\label{*IntegrabilityDistributionsDLLocalKoszulBanachBundles}Integrability of distributions which are direct limit of local Koszul Banach
bundles}

\subsection{\label{**IntegrabilityBanach}Integrability of the range of an anchor}

We first recall the classical definitions of distribution,
integrability and involutivity.

\begin{definition}
${}$ Let $M$ be a Banach manifold.

\begin{enumerate}
\item A distribution $\Delta$ on $M$ is an assignment $\Delta: x\mapsto
\Delta_{x}\subset T_{x}M$ on $M$ where $\Delta_{x}$ is a subspace of $T_{x}M$.

\item A vector field $X$ on $M$, defined on an open set Dom$(X)$, is called
tangent to a distribution $\Delta$ if $X(x)$ belongs to $\Delta_{x}$ for all
$x\in$Dom$(X)$.

\item A distribution $\Delta$ on $M$ is called integrable if, for all
$x_{0}\in M$, there exists a weak submanifold $(N,\phi)$ of $M$ such that
$\phi(y_{0})=x_{0}$ for some $y_{0}\in N$ and $T\phi(T_{y}N)=\Delta_{\phi(y)}$
for all $y\in N$. In this case $(N,\phi)$ is called an integral manifold of
$\Delta$ through $x$.

\item A distribution $\Delta$ is called involutive if for any vector fields $X$
and $Y$ on $M$ tangent to $\Delta$ the Lie bracket $[X,Y]$ defined on
Dom$(X)\cap$Dom$(Y)$ is tangent to $\Delta$.
\end{enumerate}
\end{definition}

%Clearly, an integrable distribution is involutive, but the converse is not
%true in general (even in a Banach context).

Classically, in Banach context, when $\Delta$ is a complemented subbundle of
$TM$, according to the Frobenius Theorem, involutivity implies integrability.

In finite dimension, the famous results of H. Sussman and P. Stefan give
necessary and sufficient conditions for the integrability of smooth distributions.

A generalization of these results in the context of Banach manifolds can be
found in \cite{ChSt} and \cite{Pel}.

We are now in a position to prove the following theorem which will be useful
for the proof of the main theorem on the integrability of a distribution on a
direct limit of Banach manifolds endowed with Koszul connections.

\begin{theorem}
\label{T_ConnectionDistribution} Let $(E,\pi,M,\rho,[\;,\;]_{E})$ be a Banach
algebroid (cf. subsection \ref{**Algebroid}). Assume that for each $x\in M$, the kernel of $\rho_x$ is
complemented in each fiber $E_x$ and  $\mathcal{D}_x=\rho(E_x)$ is  closed in $T_xM$. Then  $\mathcal{D}$ is an
integrable weak distribution of $M$. \newline Assume that there exists a
linear connection on $E$. Then there exists a non linear connection on the
tangent bundle of each leaf of the distribution $\mathcal{D}$.
\end{theorem}

\begin{proof}
The first part of this theorem is an easy adaptation of the proof of Theorem 5
in \cite{Pel}.

We consider a leaf $L$ of $\mathcal{D}$. If $\iota:L\rightarrow M$ is the
natural inclusion, it is a smooth immersion of $L$ in $M$. Let $x$ be any
point of $L$ and denote $K_{x}$ the kernel of $\rho_{x}:\pi^{-1}%
(x)=E_{x}\rightarrow T_{x}M$. According to the assumption on $E$, we have a
decomposition $E_{x}=K_{x}\oplus F_{x}$. From the proof of Theorem 2 of
\cite{Pel}, $L$ is a Banach manifold modeled on $F:=F_{x}$. Consider the pull
back $E_{L}$ of $E$ over $L$ via $\iota:L\rightarrow M$. We have a bundle
morphism $\hat{\iota}$ from $E_{L}$ in $E$ over $\iota$ which is an
isomorphism on each fiber. Therefore, the kernel of $\hat{\rho}=\rho\circ
\iota$ is a Banach subbundle $K_{L}$ of $E_{L}$ and we have a subbundle
$F_{L}$ of $E_{L}$ such that $E_{L}=K_{L}\oplus F_{L}$. In particular, we have
an isomorphism $\rho_{L}$ from $F_{L}$ to $TL$. It follows that the tangent
map $T\rho_{L}:TF_{L}\rightarrow T(TL)$ is also an isomorphism. On the other
hand, according to the decomposition $E_{L}=K_{L}\oplus F_{L}$, we have also a
decomposition $TE_{L}=TK_{L}\oplus TF_{L}$.\newline

Now, assume that there exists a non linear connection on $E$ and let
$D:TE\rightarrow E$ be the associated map connection. The map $\hat{D}%
_{L}=D\circ\hat{\iota}\circ(T\rho_{L})^{-1}$ is smooth and maps the fiber of
$T_{(x,u)}(TL)$ over $(x,u)\in TL$ into the fiber of $(E_{L})_{x}$ over $x\in
L$. As $\hat{\iota}$ is an isomorphism from $(E_{L})_{x}$ to the fiber $E_{x}$
over $\iota(x)$, it follows that $\hat{D}_{L}$ is a linear continuous map
between these fibers. In particular, we can consider $\hat{D}_{L}$ as a map
from $T(TL)$ into $E_{L}$. Now if $\Pi_{L}$ is the projection of $E_{L}$ on
$F_{L}$ parallel to $K_{L}$, the map ${D}_{L}=\rho_{L}\circ\Pi_{L}\circ\hat
{D}_{L}$ defines a Koszul connection on $TL$.
\end{proof}

\subsection{Criterion of integrability for local direct limits of local Koszul
Banach bundles}

%\begin{definition}
${}$ Let $M$ be a n.n.H. convenient manifold and denote by $TM$ its dynamical tangent
bundle. In the same way, a distribution $\Delta$ on $M$ is again an assignment
$\Delta:x\mapsto\Delta_{x}\subset T_{x}M$ on $M$ where $\Delta_{x}$ is a
subspace of $T_{x}M$. The notion of integrability and involutivity of a
distribution recalled in Subsection \ref{**IntegrabilityBanach} can be clearly
adapted to the convenient context.

%\item A vector field $X$ on $M$, defined on an open set Dom$(X)$, is called
%tangent to a distribution $\Delta$ if $X(x)$ belongs to $\Delta_{x}$ for all
%$x\in$Dom$(X)$.

%\item A distribution $\Delta$ on $M$ is called integrable if, for all
%$x_{0}\in M$, there exists a convenient weak submanifold $(N,\phi)$ of $M$
%such that $\phi(y_{0})=x_{0}$ for some $y_{0}\in N$ and $T\phi(T_{y}%
%N)=\Delta_{\phi(y)}$ for all $y\in N$. In this case $(N,\phi)$ is called an
%integral manifold of $\Delta$ through $x$.

%\item A distribution $\Delta$ is called involutive if for any vector fields
%$X$ and $Y$ on $M$ tangent to $\Delta$ the Lie bracket $[X,Y]$ defined on
%Dom$(X)\cap$Dom$(Y)$ is tangent to $\Delta$.
%\end{enumerate}
%\end{definition}

We will now give a criterion of integrability for direct limit of local Koszul
Banach bundles. More precisely we have:

\begin{definition}
A distribution $\Delta$ on a n.n.H. convenient manifold $M$ is called a local
direct limit of local Koszul Banach bundles if the following property is satisfied:

(*) for any $x\in M$, there exists an open neighbourhood $U$ of $x$ and a strong ascending
sequence of anchored Banach bundles $(E_{n},\pi_{n},U_{n},\rho_{n})_{n\in
\mathbb{N}^{\ast}}$ endowed with a Koszul connection $\nabla^{n}$ such that
$U=\underrightarrow{\lim}U_{n},$ $\underrightarrow{\lim}\rho_{n}(E_{n}%
)=\Delta_{|U}$ and such that $E_{n}$ is a complemented subbundle of $E_{n+1}.$
\end{definition}

\begin{remark}
\label{R_HilbertCont} In the context of paracompact finite dimensional
manifolds or Hilbert manifolds, the condition of the existence of a Koszul
connection $\nabla^{n}$ and $E_{n}$ complemented in $E_{n+1}$ are
automatically satisfied.
\end{remark}

We then have the following criterion of integrability:

\begin{theorem}
\label{T_IntegrabilityDLKoszulBanachBundles} Let $\Delta$ be a local direct
limit of local Koszul Banach bundles. Assume that in the property (*) there
exists an almost Lie bracket $[\;,\;]_{n}$ on $(E_{n},\pi_{n},U_{n},\rho_{n})$
such that $(E_{n},\pi_{n},U_{n},\rho_{n},[\;,\;]_{n})$ is a Banach algebroid,
and over each point $y_{n}\in U_{n}$ the kernel of $\rho_{n}$ is complemented
in the fiber $\pi_{n}^{-1}(y_{n})$ and the range of $\rho_n$ is closed. 
\newline Then the distribution $\Delta$ is
integrable and the maximal integral manifold $N$ through $x=\underrightarrow
{\lim}x_{n}$ is a weak n.n.H.  convenient submanifold of $M$ which is a direct limit
of the set of maximal leaves $N_{n}$ of $\rho_{n}(E_{n})$ through $x_{n}$ in
$M_{n}$. Moreover, each maximal leaf has the limit chart property at any point and if $M$ is Hausdorff so is each leaf.
\end{theorem}

\begin{proof}
At first, for each $n\in\mathbb{N}^{\ast}$, we can apply the first part of
Theorem \ref{T_ConnectionDistribution}. Therefore, with the notation of
property (*), if we fix some $x=\underrightarrow{\lim}x_{n}$, there exists a
maximal integral manifold $N_{n}$ of $\rho_{n}(E_{n})$ through $x_{n}$ in
$U_{n}$. Recall that we have $U_{n}\subset U_{n+1}$ and $E_{n}\subset E_{n+1}$
over $U_{n}$. Therefore, according to Property (2) of Definition
\ref{D_DirectSequenceBanachLieAlgebroids}, for any $y\in N_{n}$, we have
$T_{y}N_{n}\subset T_{y}N_{n+1}$ on $N_{n}\cap N_{n+1}$. Since $N_{n+1}$ is a
maximal integral manifold of $\rho_{n+1}(E_{n+1})$ in $U_{n+1}$ and
$U_{n}\subset U_{n+1}$, if $y$ belongs to $N_{n}$, we have a smooth curve in
$N_{n}$ which joins $x_{n}$ to $y$ and since $E_{n}\subset E_{n+1}$ over
$U_{n}$ this curve must be contained in $N_{n+1}$ and so $N_{n}$ must be
contained in $N_{n+1}$. Now, on the one hand, over each point of $U_{n}$ the
kernel of $\rho_{n}$ is complemented in each fiber and, on the other hand,
over $N_{n}$ the kernel of $\rho_{n}$ is a subbundle of ${E_{n}}_{|N_{n}}$.
The same property is true for $E_{n+1|N_{n+1}}$. But over $N_{n}\subset
N_{n+1}$, we have
\[
\rho_{n}({E_{n}}_{|N_{n}})=TN_{n}\subset\left(  TN_{n+1}\right)  _{|N_{n}%
}=\rho_{n+1}({E_{n+1}}_{|N_{n}}).
\]
Therefore $\left(  \ker\rho_{n+1}\right)  _{|N_{n}}\subset\left(  \ker\rho
_{n}\right)  _{|N_{n}}$. But, from our assumption, we have the following
Whitney decomposition:
\[
{E_{n+1}}_{|N_{n}}=F_{n+1}\oplus\left(  \ker\rho_{n+1}\right)  _{|N_{n}%
}\text{ and }{E_{n}}_{|N_{n}}=F_{n}\oplus\left(  \ker\rho_{n}\right)
_{|N_{n}}.
\]

Therefore
\[
\left(  \ker\rho_{n}\right)  _{|N_{n}}=\left(  \ker\rho_{n+1}\right)
_{|N_{n}}\oplus F_{n+1}\cap\left(  \ker\rho_{n}\right)  _{|N_{n}}
\]
Finally we obtain:
\[
\left(  TN_{n+1}\right)  _{|N_{n}}=TN_{n}\oplus\rho_{n+1}(F_{n+1}\cap\left(
\ker\rho_{n}\right)  _{|N_{n}}).
\]

Now, from property (*) and the second part of Theorem
\ref{T_ConnectionDistribution}, we have a linear connection on $TN_{n}$. Thus
the ascending sequence $(N_{n})$ satisfies the assumption of Corollary
\ref{ascendingSequenceConnection}, $N=\underrightarrow{\lim}N_{n}$ has a
structure of convenient manifold modeled on an LB-space. Moreover, by
construction, we have $TN=\Delta_{|N}$. This means that $\Delta$ is an
integral manifold of $\Delta$ through $x$. Moreover, $N$ satisfies the direct
limit chart property. \newline Take any maximal integral manifold $L$ of
$\Delta$ and choose some $x=\underrightarrow{\lim}x_{n}$ in $L$. From our
previous construction we have a sequence of Banach integral manifolds
$(N_{n})$ such that $N=\underrightarrow{\lim}N_{n}$ is an integral manifold of
$\Delta$ through $x$. Therefore $N$ is open in $L$. Since $N$ has the direct
limit chart property, the same is true of $L$.\\
Now as the intersection of an open set in $M$ with any leaf $L$ is an open set of $L$, thus, if $M$ is an Hausdorff topological space,
$L$ inherits of this property.
\end{proof}

From this result we easily obtain:

\begin{corollary}
\label{C_DLLeaf} Let $\left(  E_{n},\pi_{n},M_{n},\rho_{n},[\;,\;]_{n}\right)
_{n\in\mathbb{N}^{\ast}}$ be a strong ascending sequence of Banach algebroids provided
with a Koszul connection on each $E_{n}$ such that over each point $x_{n}\in
M_{n}$ the kernel of $\rho_{n}$ is complemented in the fiber $\pi_{n}%
^{-1}(x_{n})$ and the range of $\rho_n$ is closed. Then $\Delta=\underrightarrow{\lim}\rho_{n}(E_{n})$ is an
integrable distribution on $M=\underrightarrow{\lim}M_{n}$. Moreover, for any
$x=\underrightarrow{\lim}x_{n}$, the maximal leaf through $x$ is a weak
n.n.H. convenient submanifold of $M$ and there exists a leaf $N_{n}$ of $\rho_{n}%
(E_{n})$ in $M_{n}$ through $x_{n}$ such that the sequence $(N_{n}%
)_{n\in\mathbb{N}^{\ast}}$ is an ascending  sequence of Banach manifolds whose
direct limit $N=\underrightarrow{\lim}N_{n}$ is an integral manifold of
$\Delta$ through $x$ such that $N$ has the direct limit chart property at $x$. Moreover, if $M$ is Hausdorff so is each leaf.
\end{corollary}

Now according to Remark \ref{R_HilbertCont} we also easily obtain:

\begin{corollary}
\label{C_DLHilbert} Let $\Delta$ be a distribution on a direct limit
$M=\underrightarrow{\lim}M_{n}$ of finite dimensional (resp. Hilbert)
  paracompact manifolds. Assume that, for any $x=\underrightarrow{\lim}x_{n}$,
there exists a sequence of finite rank (resp. Hilbert) algebroids $(E_{n}%
,\pi_{n},U_{n},\rho_{n})_{n\in\mathbb{N}^{\ast}}$ such that
$U=\underrightarrow{\lim}U_{n},$ $\underrightarrow{\lim}\rho_{n}(E_{n}%
)=\Delta_{|U}$. Then $\Delta$ is integrable and the maximal integral manifold
$N$ through $x=\underrightarrow{\lim}x_{n}$ is a weak convenient submanifold
of $M$ which is the direct limit of the set of maximal leaves $N_{n}$ of
$\rho_{n}(E_{N})$ through $x_{n}$ in $M_{n}$. Moreover, each maximal leaf has
the limit chart property at any point and is a Hausdorff convenient manifold.
\end{corollary}

\subsection{Application}

Consider a direct  sequence of Banach Lie groups $G_{1}\subset G_{2}%
\subset\cdots\subset G_{n}\subset\cdots$ such that the Lie algebra
$\mathcal{G}_{n}$ is complemented in the Lie algebra $\mathcal{G}_{n+1}$ for
all $n\in\mathbb{N}^{\ast}$.\newline

Note that this situation always occurs if for all $n\in\mathbb{N}^{\ast}$,
each Lie group $G_{n}$ is finite dimensional or is a Hilbert Lie group. This
assumption is also valid for the sequence $G_{n}=GL(E_{n})$ where $E_{n}$ is
a direct  sequence of Banach spaces such that each $E_{n}$ is closed and
complemented in $E_{n+1}$.\newline

According to Example \ref{Ex_StructureLieAlgebroidLinkedActionBanachLieGroup},
assume that for each $n\in\mathbb{N}^{\ast}$, we have a smooth right action
$\psi_{n}:M_{n}\times G_{n}\longrightarrow M_{n}$ of $G_{n}$ over a Banach
manifold $M_{n}$ where $M_{1}\subset M_{2}\subset\cdots\subset M_{n}%
\subset\cdots$ is an ascending sequence such that $M_{n}$ is a Banach
submanifold of $M_{n+1}$. We get a strong ascending   sequence  Lie Banach algebroids
$(M_{n}\times\mathcal{G}_{n},\pi_{n},M_{n},\Psi_{n},[\;,\;]_{\mathcal{G}_{n}%
})$. Since each Banach bundle $M_{n}\times\mathcal{G}_{n}$ is trivial, we
obtain a sequence of Banach Lie algebroids with anchors 
\[
\begin{array}
[c]{cccc}%
\Psi_{n}: & M_{n}\times\mathcal{G}_{n} & \longrightarrow & TM_{n}\\
& \left(  x_{n},X_{n}\right)  & \mapsto & T_{\left(  x_{n},e_{n}\right)}%
\psi_{n}\left(  0,X_{n}\right)
\end{array}
\]
Because these bundles are trivial, we get a sequence of compatible trivial
Koszul connections $\nabla^{n}$ on $M_{n}\times\mathcal{G}_{n}$. Now, from the
Corollary \ref{C_DLLeaf}, we obtain:

\begin{theorem}
\label{T_DirectLimitComplementedAction} In the previous context, we obtain a
smooth right action $\psi=\underrightarrow{\lim}\psi_{n}$ of $G=\underrightarrow
{\lim}G_{n}$ on the convenient manifold $M=\underrightarrow{\lim}M_{n}$.
Moreover, if the kernel of $\Psi_{n}$ is complemented in each fiber $\pi
_{n}^{-1}(x_{n})$ and the range of $\Psi_n$ is closed, then the orbit $\psi(x,G)$ of this action through
$x=\underrightarrow{\lim}x_{n}$ is a weak n.n.H. convenient submanifold of $M$ which
is the direct limit of the set of $G_{n}$-orbits $\{\psi_{n}(x_{n},G_{n}%
)\}_{n\in\mathbb{N}^{\ast}}$. If $M$ is Hausdorff, so is each  orbit.
\end{theorem}

\end{document}